\documentclass[12pt, a4paper, reqno]{amsart}
\usepackage{tikz}
\usetikzlibrary{arrows.meta, positioning, fit}
\usepackage{amssymb,amsmath,amsfonts,amsthm}
\usepackage{yhmath}
\usepackage{url}
\usepackage{preamble}
\usepackage{tikz-cd}
\usepackage{color}
\usepackage{todonotes}

\allowdisplaybreaks

\newcommand{\td}{\mathrm{d}}
\newcommand{\gs}{\geqslant}
\newcommand{\ls}{\leqslant}
\newcommand{\e}{\mathrm{e}}

\newcommand{\M}{\mathcal{M}}

\newcommand{\R}{\mathbb{R}}
\newcommand{\N}{\mathbb{N}}

\newcommand{\V}{\mathcal{V}}

\newtheorem{corollary}[theorem]{Corollary}

\def\vep{\epsilon}

\def\supp{\operatorname{supp}}

\def\top{\operatorname{top}}
\def\erg{\operatorname{erg}}
\def\H{\mathcal{H}}
\def\i{\mathrm{i}}
\def\vep{\varepsilon}

\title
[A prime orbit theorem]
{A prime orbit theorem for smooth surface diffeomorphisms}
\begin{document}
\begin{abstract}
We establish a sharp prime orbit theorem for every homoclinic class of a $C^\infty$ diffeomorphism on a closed surface with positive topological entropy. Let $\mathcal{H}$ be a homoclinic class with topological entropy $h > 0$. Then there exists a constant $\chi_2 < 0$ such that for any $\chi_1 \in (0, h)$,
\[
\lim_{\substack{l(\mathcal{H}) \mid n \\ n\to\infty}}
\frac{\sharp P_{\chi_1,\chi_2}(n)}{e^{nh}}
= l(\mathcal{H}).
\]
Here $P_{\chi_1,\chi_2}(n)$ stands for the set of period-$n$ saddle points in $\mathcal{H}$ with Lyapunov exponents lying outside the interval $[\chi_2,\chi_1]$, and $l(\mathcal{H})$ denotes the period associated with the homoclinic class $\mathcal{H}$.

\end{abstract}

\maketitle

\tableofcontents

\section{Introduction}\label{sec:intro}

In dynamical systems, periodic orbits encode the topological and statistical complexity of underlying dynamics, serving as {basic building blocks of recurrent dynamics, invariant measures, entropy theory, and geometric rigidity}. The result characterizing their asymptotic growth is known as the \textit{prime orbit theorem}, which draws a deep and precise analogy with the prime number theorem in number theory: periodic orbits of a chaotic dynamical system grow exponentially at a rate equal to the topological entropy, with a sharp leading-order constant determined by the system's mixing structure.

The classical theory began with the prime geodesic theorem. Huber \cite{Huber61}, building on Selberg's trace formula \cite{Se56}, proved the theorem for compact surfaces of constant negative curvature. Margulis \cite{Mar69,Mar} later established the corresponding asymptotic for compact manifolds of variable negative curvature by dynamical methods. Bowen developed periodic orbit counting and equidistribution for hyperbolic flows \cite{Bowen72-2,Bowen72-1}, and Parry and Pollicott proved an analogue of the prime number theorem for Axiom A flows \cite{PP83}. In the decades that followed, substantial progress was made on error terms, equidistribution, homology classes, and generalizations to rank-one spaces; see the surveys in \cite{Kni, Pol}, and references therein for a comprehensive account of the classical theory.

For uniformly hyperbolic diffeomorphisms, the analytic and symbolic foundations were developed through Markov partitions \cite{Bowen1970Markov}, the Artin--Mazur zeta function \cite{ArtinMazur1965}, Bowen's periodic-point theory \cite{Bowen1971}, Manning's rationality theorem \cite{Manning1971}, and Ruelle's work on dynamical zeta functions and thermodynamic formalism \cite{Ruelle1976,Ruelle.3}.
In particular, for a mixing Axiom A diffeomorphism,
\[
\lim_{n\to\infty}
\frac{\sharp\{n\text{-periodic points}\}}
{\e^{n h_{\top}}}
=1,
\]
where \(h_{\top}\) denotes the topological entropy of the system. For a transitive basic component, the corresponding asymptotic is taken along multiples of its period, with the period appearing as the leading constant.

Despite the success of the uniformly hyperbolic theory, many natural dynamical systems, including H\'enon attractors, standard maps, and generic surface diffeomorphisms, exhibit only \textit{non-uniform hyperbolicity}: Lyapunov exponents are non-zero almost everywhere with respect to a natural invariant measure, but there is no uniform expansion or contraction rate across the entire phase space. 
Extending the prime orbit theorem to this general setting is more delicate. A prime orbit theorem for all periodic points cannot hold for general smooth systems: Kaloshin \cite{kaloshin2000} showed that generic surface diffeomorphisms may have superexponential growth of periodic points. It is therefore natural to impose a uniform gap between the Lyapunov exponents and zero.

Katok's approximation theorem \cite{katok1980} gives the lower exponential growth rate of hyperbolic periodic points for surface diffeomorphisms with positive entropy; see also the discussion and conjectures in \cite{Katokcon}. For periodic points whose Lyapunov exponents are uniformly bounded away from zero, Burguet \cite{Bur} proved asymptotical periodic expansiveness and obtained the matching upper bound at the exponential scale. These results determine the exponential growth rate, but not the exact leading coefficient.

A recent breakthrough toward exact counting was provided by symbolic dynamics for non-uniformly hyperbolic surface diffeomorphisms. Sarig \cite{Sarig13} constructed countable Markov codings that capture all sufficiently hyperbolic invariant measures, making the thermodynamic formalism available for countable Markov shifts \cite{OS,OS.1}. Buzzi, Crovisier, and Sarig \cite{BCS} subsequently established the finiteness and homoclinic organization of measures of maximal entropy, and their later work developed continuity properties of Lyapunov exponents \cite{bcs2022}. Buzzi \cite{buzzi2020} studied Bowen quotient codings and obtained the optimal lower bound for hyperbolic periodic points, with the precise leading coefficient. The strong positive recurrence theory developed in \cite{SPR2025} supplies the recurrence and entropy-tightness properties for surface diffeomorphisms.

A remaining difficulty in obtaining an exact upper asymptotic is that the available symbolic codings are, in general, finite-to-one rather than globally injective. While finite-to-one coding is sufficient for entropy and measure-theoretic considerations, it does not by itself preserve the exact leading coefficient in periodic-point counting.

In this paper, we develop the periodic-orbit counting in the non-uniformly hyperbolic setting by establishing the \textit{matching upper bound with an exact leading constant}, thus obtaining a prime orbit theorem for every positive-entropy homoclinic class of a \(C^\infty\) surface diffeomorphism.

Given \(\chi_1>0>\chi_2\), a \((\chi_1,\chi_2)\)-hyperbolic periodic point is a saddle periodic point whose Lyapunov exponents lie outside the interval \([\chi_2,\chi_1]\). Given a periodic point \(x\), denote its orbit by \(\mathcal{O}(x)\). For any subset \(A\subset M\) and \(n\in\mathbb{N}\), define
\[
P_{\chi_1,\chi_2}(f,A,n)
:=
\left\{
x\in A:
\sharp\mathcal{O}(x)=n
\text{ and }x\text{ is }(\chi_1,\chi_2)\text{-hyperbolic}
\right\}.
\]

Homoclinic classes are natural analogues of basic sets beyond uniform hyperbolicity. Recall that the homoclinic class of a saddle periodic orbit is the closure of the union of all saddle periodic orbits homoclinically related to it.

\begin{theorem}\label{exp}
Let \(f\) be a \(C^\infty\) diffeomorphism on a closed surface \(M\), and let \(\H\) be a homoclinic class with positive topological entropy. Then there exists \(\chi_2<0\) such that for any
\(
\chi_1\in\bigl(0,h_{\top}(f,\H)\bigr),
\)
one has
\[
\lim_{\substack{l(\H)\mid n\\ n\to\infty}}
\frac{\sharp P_{\chi_1,\chi_2}(f,\H,n)}
{\e^{n h_{\top}(f,\H)}}
=
l(\H),
\]
where \(l(\H)\) is the period of the homoclinic class \(\H\).
\end{theorem}

It is worth noting that the negative threshold \(\chi_2\) depends only on the homoclinic class and can be chosen simultaneously for every
\(
\chi_1\in\bigl(0,h_{\top}(f,\H)\bigr).
\)

Newhouse \cite{New89} proved that \(C^\infty\) diffeomorphisms on closed manifolds admit measures of maximal entropy. For surface diffeomorphisms with positive entropy, Buzzi, Crovisier, and Sarig \cite{BCS} proved that there are finitely many ergodic measures of maximal entropy, and exactly one in the topologically transitive case.

When the whole system is transitive, we use the shorthand
\[
P_{\chi_1,\chi_2}(f,n)
=
P_{\chi_1,\chi_2}(f,M,n).
\]

\begin{corollary}\label{transitive case}
Let \(f\) be a \(C^\infty\) transitive diffeomorphism on a closed surface \(M\) with positive topological entropy. Then there exist \(l\in\mathbb{N}\) and \(\chi_2<0\) such that for any
\(
\chi_1\in\bigl(0,h_{\top}(f)\bigr),
\)
one has
\[
\lim_{\substack{l\mid n\\ n\to\infty}}
\frac{\sharp P_{\chi_1,\chi_2}(f,n)}
{\e^{n h_{\top}(f)}}
=
l.
\]
In particular, if \(f\) is mixing, one can take \(l=1\).
\end{corollary}

\begin{remark}
The mechanism in this paper is expected to yield an analogous prime orbit theorem for any \(C^\infty\) flow on closed three-manifolds with positive entropy, using corresponding estimates on transverse sections. The symbolic dynamics for positive-entropy three-dimensional flows were constructed by Lima and Sarig \cite{LimaSarig2019}, Buzzi, Crovisier, and Lima \cite{BCL}, and Li and Liu \cite{li-liu-singularflowcoding}. This extension will be treated in future work.
\end{remark}

\medskip

\noindent{\bf Strategy of proof}
\medskip

A localization of Buzzi's injective-coding and periodic-counting argument \cite[Theorem 1.2]{buzzi2020} to the irreducible component associated with \(\H\) provides the required lower bound.
The main task is, therefore, to establish the upper bound. We first work in the mixing case; the general case follows by passing to the cyclic decomposition of the homoclinic class.

\begin{enumerate}
\item[(1)] \textbf{Periodic points away from the MME.}
Let \(\mu\) be the unique measure of maximal entropy carried by \(\H\). We divide the relevant periodic points according to whether their orbit measures are close to \(\mu\) in the weak\(^*\) topology. For orbit measures staying a fixed distance from \(\mu\), the uniqueness of the MME gives an entropy gap. Burguet's periodic expansiveness estimate then shows that these periodic points have exponential growth strictly smaller than \(h_{\top}(f,\H)\).

\item[(2)] \textbf{Localization of periodic points close to the MME.}
Using the large-Pesin-block estimate, we choose a negative exponent threshold \(\chi_2\) uniformly and find a compact Pesin block carrying a fixed proportion of every relevant periodic orbit whose orbit measure is sufficiently close to \(\mu\). Uniform local stable and unstable manifolds on this block then imply that all such periodic orbits belong to the Borel homoclinic class carrying \(\mu\).

\item[(3)] \textbf{Symbolic coding and the magic-word decomposition.}
We use the irreducible SPR coding of Buzzi--Crovisier--Sarig and Bowen quotient. A magic word separates the symbolic periodic points into two families. Periodic points whose lifts avoid the magic cylinder satisfy a strict entropy gap. The proof combines the SPR first-return estimate with Burguet's local periodic packing bound and a separated-set argument.  

\item[(4)] \textbf{Counting periodic points that see the magic word.}
The Bowen quotient is injective on the symbolic component generated by the magic word, when seeing magic words.
Inducing on the magic cylinder yields a countable full shift over first-return words, with an integer-valued return time depending only on the zero-th symbol. The resulting zeta function is a renewal series, which gives its meromorphic extension, the simple leading poles, and the sharp periodic-orbit asymptotic.
Combining this with the entropy gap for the non-magic family gives the upper bound with leading constant \(1\) in the mixing case.
This is the countable-state counterpart of the classical transfer-operator and zeta-function approach to prime orbit theorems \cite{Pollicott1986,ParryPollicott1990}.
\end{enumerate}

\medskip

\noindent{\bf Organization of the paper}
\medskip

Section \ref{pre} recalls the necessary background from ergodic theory,
Pesin theory, strong positive recurrence, and homoclinic classes.
Section \ref{sec:away from mme} proves that periodic orbits whose
empirical measures stay away from the measure of maximal entropy are
exponentially negligible. Section \ref{sec:close to mme} introduces
the countable Markov coding and the Bowen quotient, localizes the
relevant periodic points on a Pesin block, and establishes the entropy
gap for orbits avoiding a magic word. Finally, Section \ref{see}
counts the periodic orbits visiting the magic cylinder by means of the
induced symbolic system and its dynamical zeta function.

\medskip
\noindent{\bf Acknowledgments.}
During the preparation of this manuscript, we learned that Buzzi, Crovisier and Sarig have independently derived a periodic orbit growth estimate via a different approach. We thank them for kindly communicating their progress. G. Liao was partially supported by the National Key R\&D Program of China (2022YFA1005802).
Y. Tong was partially supported by the National Key R\&D Program of China (2022YFA1005801).
\section{Preliminaries}\label{pre}

\subsection{Ergodic theory}

We first present some preliminary concepts from abstract ergodic theory. Let $T$ be an arbitrary transformation on a measurable space $(\Omega,\mathcal{F})$. Denote by $\mathcal{M}(T)$ (resp. $\mathcal{M}_{\mathrm{erg}}(T)$) the set of $T$-invariant (resp. ergodic $T$-invariant) probability Borel measures on $(\Omega,\mathcal{F})$. Given a $T$-invariant measurable subset $X \subseteq \Omega$, define
\[
\mathcal{M}(T,X):= \bigl\{\mu \in \mathcal{M}(T): \mu(X) = 1\bigr\}.
\]
For any $\mu \in \mathcal{M}(T)$, let $h_{\mu}(T)$ stand for its metric entropy. The topological entropy of $T$ restricted to $X$ is
\[
h_{\mathrm{top}}(T, X):= \sup\bigl\{h_{\mu}(T): \mu \in \mathcal{M}(T, X)\bigr\}.
\]

By the variational principle, if $\Omega$ is a compact metric space and $T$ is continuous, then $h_{\mathrm{top}}(T, \Omega)$ coincides with the global topological entropy $h_{\mathrm{top}}(T)$ of $T$; see \cite{Walter}.

An invariant measure $\mu \in \mathcal{M}(T, X)$ satisfying $h_{\mu}(T) = h_{\mathrm{top}}(T, X)$ is termed a \textit{measure of maximal entropy} (MME) over $X$.

\subsection{Pesin theory }
Let \(M\) be a closed smooth Riemannian manifold, and let \(f:M\to M\) be a \(C^1\) diffeomorphism. Throughout this section, we fix two Lyapunov rate parameters \(\chi>0\) and \(\varepsilon>0\) satisfying \(\varepsilon\ls \chi\), and we follow the standard notational framework of Pesin’s nonuniform hyperbolicity theory.

\subsubsection{Pesin block}
\begin{definition}
A nonempty compact subset \(\Lambda\subset M\) is called a \((\chi,\varepsilon)\)-Pesin block if for every point \(x\in \bigcup_{n\in\mathbb Z}f^n(\Lambda)\), the tangent bundle admits a  direct sum  decomposition
\[
T_xM = E^s(x)\oplus E^u(x).
\]
There exists a uniform constant \(k\in \mathbb{N}\) such that for all integers \(n\in\mathbb{N}\), all integers \(m\in \mathbb{Z}\), and all points \(y\in\Lambda\), the exponential decay estimate
\[
\max\left(\left\|Df^n|_{E^s(f^m(y))}\right\|,\left\|Df^{-n}|_{E^u(f^m(y))}\right\|\right) \ls e^{k\vep}\exp\left(-\chi n + \varepsilon|m|\right)
\]
holds uniformly on the forward and backward orbits of \(\Lambda\).

Let $\Lambda_k(\chi,\vep)$ be the maximal Pesin block  satisfying the above property for fixed $k$.
\end{definition}

  The global Pesin set is the countable union of compact Pesin blocks:
\[
\Lambda^*(\chi, \vep) = \bigcup_{k\gs1}\Lambda_k(\chi,\vep),
\]
where each \(\Lambda_k(\chi, \epsilon)\) constitutes a standard \((\chi,\varepsilon)\)-Pesin block. 
On each Pesin block, the stable and unstable subbundles \(E^s\) and \(E^u\) are continuous, and they are invariant under the diffeomorphism \(f\) on the global Pesin set.

\subsubsection{Stable/unstable manifolds}
We recall the stable manifold theorem for nonuniformly hyperbolic points.

\begin{theorem}[Pesin \cite{pesin1976}]
Let $M$ be a closed Riemannian manifold and fix $r>1$. There exists a positive continuous function $\varepsilon_{M,r}$ such that the following holds.

For any $C^r$ diffeomorphism $f$ of $M$, any pair $\chi,\varepsilon>0$ with $\varepsilon<{\varepsilon_{M,r}(\chi,\|f\|_{C^r})}$, and any $(\chi,\varepsilon)$-Pesin block $\Lambda$, each point $x\in\Lambda$ possesses a $C^r$ embedded local stable disc $W^s_{\mathrm{loc}}(x)$ satisfying
\[
\forall y\in W^s_{\mathrm{loc}}(x),\quad \limsup_{n\to\infty}\frac1n\log d(f^n(x),f^n(y))<0.
\]

{Furthermore, the local stable discs $W^s_{\mathrm{loc}}(x)$ depend continuously on $x\in\Lambda$ in the $C^1$ topology. An identical statement holds for local unstable manifolds $W^u_{\mathrm{loc}}(x)$.}
\end{theorem}

{
For a Lyapunov regular point $x$, the global stable and unstable manifolds are defined by
\[
W^s(x):=\left\{y\in M:\limsup_{n\to\infty}\frac{1}{n}\log d(f^n(x),f^n(y))<0\right\},
\]
\[
W^u(x):=\left\{y\in M:\limsup_{n\to\infty}\frac{1}{n}\log d(f^{-n}(x),f^{-n}(y))<0\right\}.
\]
Equivalently, using the local Pesin manifolds along the orbit of $x$, one has
\[
W^s(x)=\bigcup_{n\gs0}f^{-n}\big[W^s_{\mathrm{loc}}(f^n(x))\big],\quad
W^u(x)=\bigcup_{n\gs0}f^n\big[W^u_{\mathrm{loc}}(f^{-n}(x))\big].
\]
}

\begin{definition}
Define the fully nonuniformly hyperbolic set
\[
\mathrm{NUH}(f):=\bigcup_{\chi>0}\bigcap_{\varepsilon>0}\big\{x: x\text{ is contained in some }(\chi,\varepsilon)\text{-Pesin block}\big\}. \tag{2.2}
\]
The recurrent nonuniformly hyperbolic set is
\[
\mathrm{NUH}'(f):=\bigcup_{\chi>0}\bigcap_{\varepsilon>0}\bigcup_{\substack{\Lambda\\(\chi,\varepsilon)\text{-Pesin block}}}\big\{x: x\text{ is a limit of periodic points lying in }\Lambda\big\}. 
\]
\end{definition}

For later usage with a fixed hyperbolicity rate $\chi>0$, we set
\[
\mathrm{NUH}_\chi(f):=\bigcap_{\varepsilon>0}\big\{x: x\text{ is contained in a }(\chi,\varepsilon)\text{-Pesin block}\big\}
\]
and define $\mathrm{NUH}'_\chi(f)$ analogously.

\begin{theorem}[{Katok \cite{katok1980}}]
Let $f$ be a $C^{1+}$ diffeomorphism on a closed manifold. Every hyperbolic $f$-invariant ergodic measure satisfies $\nu(\mathrm{NUH}'(f))=1$.
\end{theorem}

\subsection{Strong positive recurrence}

We now further state a strong form of non-uniform hyperbolicity:  the Strong Positive Recurrence (SPR) property of diffeomorphisms, introduced by Buzzi, Crovisier and Sarig in \cite{SPR2025}, which demands that all ergodic measures with entropy exceeding a prescribed threshold assign a uniformly bounded positive mass to a fixed Pesin block:

\begin{definition}[Buzzi, Crovisier and Sarig  \cite{SPR2025}]\label{SPR}
A diffeomorphism $f$ is said to be Strongly Positively Recurrent (SPR) on an invariant Borel subset $X$ if there exists $\chi>0$ such that for every $\varepsilon>0$, one can find a $(\chi,\varepsilon)$-Pesin block $\Lambda$,  constants $h_0 < h_{\mathrm{top}}(f, X)$ and $\tau>0$ such that:
\[
\forall\,\nu\in \mathcal{M}_{\erg}(f, X):\quad h_{\nu}(f) > h_0 \implies \nu(\Lambda) > \tau.
\]
In this setting, $X$ is called an SPR subset for $f$.
\end{definition}


\subsection{Homoclinic classes}\label{hom}

In the setting of  Theorem \ref{exp}, we consider a homoclinic class $\H$. To set the stage for the proof, we first recall the precise definition of homoclinic classes. Let $f$ be a $C^{\infty}$  diffeomorphism on a closed Riemannian surface $M$. Denote by $\mathcal{M}^{*}(f)$  the set of all saddle type hyperbolic, $f$-invariant ergodic measures.  It holds that for any $\mu \in \mathcal{M}^{*}(f)$, almost every $x$ belongs to $\mathrm{NUH}(f)$  thus possesses stable and unstable manifolds.

\begin{definition}[Homoclinically Related]\label{def homoclinic related}
For $\mu_{1}, \mu_{2} \in \mathcal{M}^{*}(f)$, we say $\mu_{1}$ and $\mu_{2}$  are homoclinically related if there exist measurable sets $A_{1}, A_{2} \subset M $ with $\mu_{i}(A_{i}) > 0$ such that for all $(x_{1}, x_{2}) \in A_{1} \times A_{2}$, the manifolds $W^{u}(x_{1}) $ and $W^{s}(x_{2})$  have a point of transverse intersection; and similarly, there exist sets $B_{1}, B_{2} \subset M$ with $\mu_{i}(B_{i}) > 0$  such that for all $(y_{1}, y_{2}) \in B_{1} \times B_{2}$, the manifolds $W^{s}(y_{1})$  and $W^{u}(y_{2})$  have a point of transverse intersection. We denote this relation by $\mu_{1} \overset{h}{\sim} \mu_{2}$.
\end{definition}

The homoclinic relation of measures is an equivalence relation (Proposition 2.11 of \cite{BCS}). Note that each periodic orbit $\mathcal{O}$  carries a unique invariant ergodic measure $\mu_{\mathcal{O}}$. We write $\mathcal{O} \overset{h}{\sim} \mu$ when $\mu_{\mathcal{O}} \overset{h}{\sim} \mu$.

Denote
$$\operatorname{Per}_{h}(f) := \{\mathcal{O}: \mathcal{O} \text{ is a hyperbolic periodic orbit of saddle type}\}.$$
For two orbits $\mathcal{O}_1, \mathcal{O}_2 \in \operatorname{Per}_{h}(f)$, the homoclinic relation $\mathcal{O}_1 \overset{h}{\sim} \mathcal{O}_2$ holds if and only if $\mu_{\mathcal{O}_1} \overset{h}{\sim} \mu_{\mathcal{O}_2}$. Then we can define the homoclinic class.

\begin{definition}[Homoclinic class]
The homoclinic class of $\mathcal{O}\in\operatorname{Per}_{h}(f)$ is the set

{
$$\H(\mathcal{O}) := \overline{\left\{\mathcal{O}' \in \operatorname{Per}_{h}(f) : \mathcal{O}' \overset{h}{\sim} \mathcal{O}\right\}}.
$$
}
The integer $$\gcd\{\sharp\mathcal{O}' : \mathcal{O}' \in \operatorname{Per}_{h}(f),\,\, \mathcal{O}' \overset{h}{\sim} \mathcal{O}\}$$ is called the period of the homoclinic class, denoted by $l(\H(\mathcal{O}))$, where $\sharp A $ is the cardinality of a set $A$.
\end{definition}

{By the Spectral Decomposition Theorem, see for instance \cite[Theorem 1]{BCS}, there exists a compact set $A$ such that for $l=l(\H(\mathcal{O}))$, one has $\H(\mathcal{O})=\cup_{i=0}^{l-1} f^i(A)$, $f^{l}(A)=A$, and $f^{l}$ is topologically mixing on $A$.}
In what follows, we assume $f$ is mixing on $\H$ by considering its period, i.e., we take $l(\H) = 1$.
In the general setting, we only need to take a $l(\mathcal{H})$ iteration of $f$ and count the periodic orbits in all the cyclic components.


\begin{definition} Two points $x,y \in \mathrm{NUH}'(f)$ are homoclinically related $(x \sim y)$ if $W^s(x)$ has a transverse intersection point with an iterate of $W^u(y)$, and $W^u(x)$ has a transverse intersection point with an iterate of $W^s(y)$.
    \end{definition}

By Proposition 2.6 of \cite{SPR2025},  the homoclinic relation $\sim$ is an equivalence relation on $\mathrm{NUH}'(f)$. 
\begin{definition}
  Borel homoclinic classes are equivalence classes of $\sim$ in $\mathrm{NUH}'(f)$. 
\end{definition}

Any hyperbolic ergodic measure is carried by a Borel homoclinic class. 
In  Theorem \ref{exp}, since $\H$ is a homoclinic class with positive topological entropy, it supports a unique measure $\mu \in \mathcal{M}^{*}(f)$  with entropy equal to $h_{\text{top}}(f, \H)$, which is called the measure of maximal entropy (MME); see \cite[Theorem 2]{BCS}.  In what follows, we denote by $\mathcal{X}$ the  (unique) 
 Borel homoclinic class that carries $\mu$.  

 The sets $\H$ and $\mathcal{X}$ describe the same homoclinic component
from two different viewpoints. Indeed, let
$\mathcal{O}\in\operatorname{Per}_{h}(f)$ be such that
$\H=\H(\mathcal{O})$. The set $\H$ is the \emph{topological homoclinic
class} of $\mathcal{O}$: it is a compact set obtained by taking the
closure of all saddle periodic orbits homoclinically related to
$\mathcal{O}$. By contrast, $\mathcal{X}$ is the \emph{Borel
homoclinic class} carrying $\mu$: it is the equivalence class, inside
$\mathrm{NUH}'(f)$, of $\mu$-typical points under the pointwise
homoclinic relation, and it need not be closed.

By \cite[Theorem 2]{BCS}, the measure $\mu$ is homoclinically related
to $\mathcal{O}$. Equivalently, $\mu$ and $\mu_{\mathcal{O}}$ are
carried by the same Borel homoclinic class; see
\cite[Remark 2.11]{SPR2025}. Hence $\mathcal{X}$ is also the Borel
homoclinic class of $\mathcal{O}$, and
\[
    \mathcal{X}\subseteq\H,
    \qquad
    \overline{\mathcal{X}}=\H,
    \qquad
    \mu(\mathcal{X})=1,
    \qquad
    \operatorname{supp}(\mu)=\H.
\]
The difference $\H\setminus\mathcal{X}$ may contain points
outside $\mathrm{NUH}'(f)$ or points belonging to other Borel
homoclinic classes, but it has zero measure for every {ergodic} invariant measure with positive entropy carried by $\H$; see
\cite[Remark 2.9]{SPR2025}. We fixed $\H$, $\mathcal{X}$ and $\mu$ in the following paper.
 
 \begin{theorem}[Theorem 3.15 of \cite{SPR2025}]\label{spr h x} $\H$ and $\mathcal{X}$ are SPR. 
      \end{theorem}

For $\eta > 0$, we define the following two sets
\begin{align*}
P_{\chi_1,\chi_2}(f, \H, \eta, n) &:= \{ x \in P_{\chi_1,\chi_2}(f, \H, n) : d_{\M}(\mu_{\mathcal{O}(x)}, \mu) > \eta \}; \\[2mm]
P^c_{\chi_1,\chi_2}(f, \H, \eta, n) &:= \{ x \in P_{\chi_1,\chi_2}(f, \H, n) : d_{\M}(\mu_{\mathcal{O}(x)}, \mu) \ls \eta \},
\end{align*}
where $d_{\mathcal{M}}$ denotes the metric on the space of Borel probability measures over $M$, which is compatible with the weak$^*$ topology.

Theorem \ref{exp} follows by combining Proposition \ref{away} and Proposition \ref{close} below. Proposition \ref{away} is proved in \Cref{sec:away from mme}, and Proposition \ref{close} is proved in \Cref{sec:close to mme}.
\begin{proposition}\label{away}
For any $\eta > 0$ and any $\chi_1 > 0 > \chi_2$,
\[
\limsup_{n \to \infty} \frac{\log \sharp P_{\chi_1, \chi_2}(f, \H, \eta, n)}{n} < h_{\top}(f, \H).
\]
\end{proposition}

\begin{proposition}\label{close}
There exists $\chi_2 < 0$ such that for any $\chi_1 \in (0, h_{\top}(f, \H)),$ there exists $\eta > 0$ with
$$\lim_{n \to \infty} \frac{\sharp P^c_{\chi_1,\chi_2}(f, \H, \eta, n)}{\e^{n h_{\top}(f, \H)}} = 1.$$

\end{proposition}

\section{Counting orbits away from the MME}\label{sec:away from mme}

We first show that periodic orbits whose associated measures are away from $\mu$ do not contribute to the main asymptotic count.

\begin{proof}[Proof of \Cref{away}]
Define the $n$-dynamical ball at $x\in M$ with radius $\delta>0$ by
\[
B_n(x, \delta) = \big\{ y \in M : \td(f^i x, f^i y) < \delta,\; 0 \ls i \ls n-1 \big\},
\]
and denote $B_{\infty}(x, \delta) = \displaystyle\bigcap_{n \gs 1} B_n(x, \delta)$.
For any subset $\mathcal{P}\subset M$, set
\[
\mathcal{P}_n = \big\{ x \in \mathcal{P} : f^n(x) = x \big\}.
\]
Further define
\begin{align*}
g_{\mathcal{P}}(\delta) &= \limsup_{n \to \infty} \frac{1}{n} \sup_{x \in M} \log \sharp\big(\mathcal{P}_n \cap B_n(x, \delta)\big), \\[2mm]
g^*_{\mathcal{P}} &= \lim_{\delta \to 0} g_{\mathcal{P}}(\delta).
\end{align*}

For $C^\infty$ surface diffeomorphisms, let $\mathcal{P}$ be the set of saddle periodic points whose Lyapunov exponents lie outside the interval $[-\chi, \chi]$ for arbitrary $\chi>0$. Burguet \cite{Bur} established the asymptotically periodic expansive property:
\[
g^*_{\mathcal{P}} = 0.
\]

{Set
\[
K_\eta:=\big\{m\in\M(f,\H):d_{\M}(m,\mu)\gs\eta\big\}.
\]
By compactness, choose finitely many weak$^*$-open convex neighborhoods $\mathcal{B}_1,\dots,\mathcal{B}_l$ covering $K_\eta$ such that
\[
\overline{\mathcal{B}}_i\cap\big\{m\in\M(f,\H):d_{\M}(m,\mu)<\eta/2\big\}=\emptyset
\]
for every $1\ls i\ls l$.} Let
\[
\mathcal{P}_{\chi_1,\chi_2, \mathcal{B}_i}
= \bigcup_{n} P_{\chi_1, \chi_2}(f, \H, \mathcal{B}_i, \eta, n),
\]
where
\[
P_{\chi_1, \chi_2}(f, \H, \mathcal{B}_i, \eta, n)
=\big\{x\in P_{\chi_1, \chi_2}(f, \H, \eta, n):\, \mu_{\mathcal{O}(x)} \in \mathcal{B}_i\big\}.
\]
It holds that
\[
g^*_{\mathcal{P}_{\chi_1, \chi_2, \mathcal{B}_i}}=0.
\]
For each fixed $\mathcal{B}_i$, choose a subsequence $\{n_k\}$ achieving the limit superior:
\[
\lim_{k \to \infty} \frac{\log \sharp P_{\chi_1, \chi_2}(f, \H, \mathcal{B}_i, \eta, n_k)}{n_k}
= \limsup_{n \to \infty} \frac{\log \sharp P_{\chi_1, \chi_2}(f, \H, \mathcal{B}_i, \eta, n)}{n}.
\]
Define empirical measures
\[
\nu_{n_k} := \frac{1}{\sharp P_{\chi_1, \chi_2}(f, \H, \mathcal{B}_i, \eta, n_k)}
\sum_{x \in P_{\chi_1, \chi_2}(f, \H, \mathcal{B}_i, \eta, n_k)} \delta_x,
\]
which admit a convergent subsequence to some $f$-invariant Borel probability measure $\nu$.
Applying \cite[Lemma 2.3]{Bur}, we obtain
\begin{equation}\label{small}
\lim_{k \to \infty} \frac{\log \sharp P_{\chi_1, \chi_2}(f, \H, \mathcal{B}_i, \eta, n_k)}{n_k}
\ls h_{\nu}(f) + g^*_{\mathcal{P}_{\chi_1,\chi_2,\mathcal{B}_i}}
= h_{\nu}(f).
\end{equation}
All periodic orbit measures induced by points in $P_{\chi_1, \chi_2}(f, \H, \mathcal{B}_i, \eta, n_k)$ belong to $\mathcal{B}_i$, so the limit measure satisfies $\nu\in \overline{\mathcal{B}}_i$, which implies
\[
d_{\M}(\nu, \mu) \gs  \frac{\eta}{2}.
\]
By the uniqueness of measure of maximal entropy (MME) within $\H$,
\[
h_{\nu}(f) < h_{\mu}(f) = h_{\top}(f, \H).
\]
Combining with \eqref{small}, we have
\[
\limsup_{n \to \infty} \frac{\log \sharp P_{\chi_1, \chi_2}(f, \H, \mathcal{B}_i, \eta, n)}{n} < h_{\mu}(f).
\]
Consequently,
\[
\begin{aligned}
\limsup_{n \to \infty} \frac{\log \sharp P_{\chi_1, \chi_2}(f, \H, \eta, n)}{n}
&\ls \limsup_{n \to \infty} \frac{1}{n}\log \Big(\displaystyle\sum_{i=1}^l \sharp P_{\chi_1, \chi_2}(f, \H, \mathcal{B}_i, \eta, n)\Big) \\[2mm]
&< h_{\mu}(f).
\end{aligned}
\]
The proof is completed.
\end{proof}

\section{Counting orbits close to the MME}\label{sec:close to mme}

To control the growth rate of the periodic points from above, we apply the hyperbolic estimates of uniform size to the symbolic codings for surface diffeomorphisms. Sarig \cite{Sarig13} constructed a countable Markov partition $\mathcal{R}$ of $ M $ and developed a corresponding symbolic dynamics. Later, Buzzi, Crovisier, and Sarig \cite{BCS} and Buzzi \cite{buzzi2020} provided refined versions of these codings, producing remarkable quantitative results for counting the number of measures of maximal entropy and periodic points.

\subsection{Symbolic dynamics}

Let $\mathcal{G}$  be a directed graph with vertex set $\mathcal{V}$  and edge set $\mathcal{E} \subset \mathcal{V} \times \mathcal{V}$. If $(V_1, V_2) \in \mathcal{E}$, we write $V_1 \to V_2$. The associated topological Markov shift is

$$\Sigma := \Sigma(\mathcal{G}) = \{ (V_i)_{i \in \mathbb{Z}} \in \mathcal{V}^{\mathbb{Z}} : V_i \to V_{i+1} \text{ for all } i \in \mathbb{Z} \}. $$
The dynamic is given by the left shift map $\sigma: \Sigma \to \Sigma$, defined by $\sigma: (v_i)_{i \in \mathbb{Z}} \mapsto (v_{i+1})_{i \in \mathbb{Z}}$, on the metric space $(\Sigma, d)$  where

$$d(\underline{v}, \underline{u}) := \exp[-\inf\{ |i| : v_i \ne u_i \}].$$

Since every vertex in the graph $\mathcal G$ constructed in \cite{Sarig13} has a finite degree, $\Sigma$ is locally compact. Moreover, $\mathcal{G}$ is irreducible, meaning that for any two vertices $V_1, V_2$ there exists a path from $V_1$ to $V_2$. The regular part of $ \Sigma $ is
$$\Sigma^{\#} := \{ (v_i) \in \Sigma : \exists\, v, w \in \mathcal{V},\, \exists\, n_k, m_k \uparrow \infty \text{ s.t. } v_{n_k} = v \text{ and } v_{-m_k} = w \},$$
which has full measure for every $\sigma$-invariant probability measure on $\Sigma$.

\subsubsection{SPR property}

For any finite sequence $a_0, a_1, \dots, a_k \in \V$, we define the cylinder set $[a_0 a_1 \cdots a_k]$ by
\[
{[a_0 a_1 \cdots a_k] = \left\{ x = (x_i)_{i\in\mathbb Z} \in \Sigma \,:\, x_i = a_i \text{ for all } 0 \ls i \ls k \right\}.}
\]
For every function $\phi: \Sigma \to \R$, denote the oscillation modulus
\[
V_n(\phi) = \sup\left\{ |\phi(x) - \phi(y)| \,:\, x_i = y_i \text{ for all } 0 \ls i \ls n-1 \right\}.
\]
A function $\phi: \Sigma \to \R$ is called \emph{locally H\"older continuous with parameter $\rho \in (0,1)$}, or simply locally H\"older continuous, if there exists a constant $C > 0$ such that
\[
V_n(\phi) \ls C \rho^n \quad \text{for all } n \gs 1.
\]
For any locally H\"older continuous potential $\phi: \Sigma \to \R$ and any fixed vertex $R \in \V$, we define the periodic orbit sums
\begin{align*}
\mathrm{F}_n([R], \phi, \sigma) &= \sum_{\substack{x \in [R] \\ \sigma^n(x) = x}} e^{\phi_n(x)}, \\[2mm]
\mathrm{F}^*_n([R], \phi, \sigma) &= \sum_{\substack{x \in [R] \\ \sigma^n(x) = x, \, H_R(x) = n}} e^{\phi_n(x)},
\end{align*}
where $\phi_n(x) = \sum_{i=0}^{n-1} \phi(\sigma^i x)$ denotes the $n$-th ergodic sum, and the return time function $H_R: \Sigma \to \N \cup \{\infty\}$ is given by
\[
H_R(x) = \mathbf{1}_{[R]}(x) \cdot \min\left\{ n \gs 1 \,:\, \sigma^n(x) \in [R] \right\}.
\]
Here $\mathbf{1}_{[R]}$ is the indicator function of the cylinder $[R]$.

The Gurevich pressure of $\phi$ is defined (\cite{OS, Gurevic}) as
\[
P_G(\Sigma, \phi) = \limsup_{n \to \infty} \frac{1}{n} \log \mathrm{F}_n([R], \phi, \sigma).
\]
This quantity is well-defined when $\sup_{\Sigma} \phi < \infty$, and the $\limsup$ reduces to a limit when $(\Sigma, \sigma)$ is topologically mixing. Moreover, it satisfies the variational principle \cite[Theorem 3]{OS}:
\[
P_G(\Sigma, \phi) = \sup\left\{ h_{\nu}(\sigma) + \int_{\Sigma} \phi\, \td \nu \,:\, \nu \in \M(\Sigma, \sigma) \right\},
\]
where $\M(\Sigma, \sigma)$ denotes the space of the $\sigma$-invariant Borel probability measures on $\Sigma$, and $h_{\nu}(\sigma)$ is the measure-theoretic entropy of $\nu$.
A measure $\nu_0 \in \M(\Sigma, \sigma)$ is called an \emph{equilibrium state} for $\phi$ if it attains the supremum in the variational principle, i.e.,
\[
h_{\nu_0}(\sigma) + \int_{\Sigma} \phi\, \td \nu_0 = P_G(\Sigma, \phi).
\]
When $\phi\equiv0$,  we have $h_{\top} (\sigma,\Sigma)=P_{G}(\Sigma, 0)$.

\begin{definition}[Sarig \cite{OS.1}]
\label{def:recurrence}
Let $\phi: \Sigma \to \R$ be a locally
H\"older continuous potential with $\sup \phi < \infty$.
\begin{enumerate}
\item $\phi$ is called \emph{positive recurrent} if and only if
\[
\sum_{n=1}^\infty e^{-n P_G(\Sigma, \phi)} \mathrm{F}_n([R], \phi, \sigma) = \infty
\quad \text{and} \quad
\sum_{n=1}^\infty n e^{-n P_G(\Sigma, \phi)} \mathrm{F}^*_n([R], \phi, \sigma) < \infty.
\]
\item $\phi$ is called \emph{strongly positive recurrent (SPR)} if and only if
\[
\limsup_{n \to \infty} \frac{1}{n} \log \mathrm{F}^*_n([R], \phi, \sigma) < P_G(\Sigma, \phi).
\]
In particular, the shift space $\Sigma$ itself is said to be SPR if the zero potential $\phi \equiv 0$ is SPR. 
\end{enumerate}
\end{definition}

\begin{proposition}[Theorem 7.4 of \cite{SPR2025}, and   \cite{Todd22entropyofTMS}]\label{spr e}
  The  following are equivalent:
  \begin{itemize}\item[(1)]$(\Sigma, \sigma)$ is SPR ;
  \item[(2)] There are a finite union $U$ of cylinders of $\Sigma$, $\tau\in (0,1)$, and $h_0<h_{\top}(\sigma,\Sigma)$ such that
  $$\forall\,\nu\in \mathcal{M}_{\erg}(\sigma, \Sigma): \,h_{\nu}(\sigma)>h_0 \Rightarrow \nu(U)>\tau.$$
  \item[(3)] There exists an MME $\mu$, and for any sequence $\{\mu_n\}$ of measures in $\mathcal{M}_{\erg}(\sigma, \Sigma)$, if 
  $\lim_{n\to \infty} h_{\mu_n}(\sigma)=h_{\top}(\sigma,\Sigma) $ then $\mu_n$ converges weak$^*$ to $\mu$. 
\end{itemize}
\end{proposition}
\subsubsection{Symbolic coding} 

We state an irreducible version of the coding, concerning the homoclinic equivalence relation of measures.  Let $\mu$ be the MME of $f$ on $\H$, and $t_1, t_2$ be its  positive and  negative Lyapunov exponents, respectively. Recall that $\mathcal X$ is the  Borel homoclinic class carrying $\mu$. {Let $\chi_0>0$ be a hyperbolicity rate in the SPR property of $\H$. We choose the coding rate $\chi$ below $\chi_0$.}  

\begin{theorem}[Theorem 3.1 of \cite{BCS}, Theorem 10.5 of \cite{SPR2025}] \label{the 5}
Let $f$ be a $C^{\infty}$  diffeomorphism on a  closed surface $M$ with positive topological entropy and $\H$ be a homoclinic class. For any $0<\chi< \min\{t_1, |t_2|\}$ and  sufficiently small $\chi'\in (0,\chi) $ there exist a locally compact countable Markov shift $(\Sigma,\sigma)$ and a H\"older continuous map $\pi: \Sigma \to M$ such that $\pi \circ \sigma = f \circ \pi$ and:
\begin{itemize}
\item[(1)] $\Sigma$ is irreducible,  strongly positive recurrent (SPR), and $h_{\top}(\sigma, \Sigma)=h_{\top}(f,\H)$. 
\item[(2)] $\pi: \Sigma^{\#} \to M$ is finite-to-one. More precisely, if $x = \pi(\underline{x})$  where $x_i = a$ for infinitely many $i < 0$ and $x_i = b$ for infinitely many $i > 0$, then $\sharp \pi^{-1}(x)$ is bounded by a constant $C(a, b)$.
\item[(3)] {$\nu(\pi(\Sigma^{\#}))=1$ for every $\chi$-hyperbolic measure $\nu \in \mathcal{M}(f,\mathcal X)$. Moreover, there exists an ergodic measure $\hat{\nu}$ on $\Sigma$ with $\pi_{*}(\hat{\nu}) = \nu$. Conversely, if $\hat{\nu}$ is an ergodic measure on $\Sigma$, then $\pi_{*}(\hat{\nu})$ is $f$-ergodic, $\chi'$-hyperbolic, homoclinically related to $\mu$, and $h_{\hat{\nu}}(\sigma) = h_{\pi_{*}(\hat{\nu})}(f)$.}
\end{itemize}
\end{theorem}

\subsection{Bowen quotient}

We further investigate injective coding via the strong positive recurrence property.
Let $(\Xi, \sigma)$ be a countable Markov systems generated by a graph $\mathcal{G}$ with vertex set $\mathcal{A}$. 

\subsubsection{Bowen property}
\begin{definition}
Consider $\pi \colon (\Xi,\sigma) \to (X,T)$ is a semiconjugacy between dynamical systems. We say $\pi$ satisfies the \textit{Bowen property} if  there exists a reflexive, symmetric binary relation $\sim$ on the alphabet $\mathcal{A}$ of $\Xi$ such that for all bi-infinite sequences $x,y\in \Xi$,
\begin{equation}\label{eq:bowen-equiv}
\pi(x) = \pi(y) \iff x_n \sim y_n \quad \text{for every } n\in\mathbb{Z}.
\end{equation}
When the above equivalence holds, we call $x,y$ \textit{Bowen equivalent} and write $x \approx y$.
The relation $\sim$ on $\mathcal{A}$ is referred to as a \textit{Bowen relation} admitted by $\pi$.
{We say $\sim$ is \textit{locally finite} if the set $\{b\in \mathcal{A} : b \sim a\}$ is finite for each symbol $a\in \mathcal{A}$.}
\end{definition}

The semiconjugacy $\pi$ appearing in Theorem \ref{the 5} admits a locally finite Bowen relation.  

\subsubsection{Bowen equivalence on words}

For a bi-infinite sequence $x\in \Xi$ and integers $a \ls b$, we adopt standard segment notation:
\[
x_{[a,b]} = x_a x_{a+1}\cdots x_b \in \mathcal{A}^{b-a+1},\qquad
x_{[a,b)} = x_{[a,b-1]}.
\]
For each integer $n>0$, the set of all admissible length-$n$ words in $\Xi$ is
\[
\mathcal{L}_n(\Xi) := \bigl\{ x_{[0,n)} : x \in \Xi \bigr\},
\]
and the full language of $\Xi$ is the union over all word lengths:
\[
\mathcal{L}(\Xi) := \bigcup_{n>0}\mathcal{L}_n(\Xi).
\]
For any word $w\in \mathcal{L}_n(\Xi)$, its length is denoted by $|w|:=n$.
We say a word $w$ \textit{occurs at position }$k\in\mathbb{Z}$ along a sequence $x\in\Xi$ if $x_{[k,k+|w|)} = w$.
Each finite word $w\in\mathcal{L}_n(\Xi)$ defines a base cylinder set:
\[
[w]_\Xi := \bigl\{ y \in \Xi : y_{[0,n)} = w \bigr\}.
\]

We say that a word $w$ occurs i.o. in $x\in \Xi$ or that $x$ sees i.o. $w$ if $w$ occurs infinitely often in both $x_{(-\infty,0]}$ and $x_{[0,\infty)}$.  A sequence $x \in \Xi$ is word recurrent if any word $w$ in $x$ occurs i.o. in $x$. We denote by $\Xi^\mathrm{rec} \subset \Xi$ the set of such sequences. Note that $\Xi^\mathrm{rec}$ carries all invariant probability measures on $\Xi$; in particular, it contains all periodic orbits. Finally, there is an obvious inclusion $\Xi^\mathrm{rec} \subset \Xi^\#$.

{
The Bowen relation $\sim$ induces a reflexive and symmetric relation (also denoted $\sim$) on the language $\mathcal{L}(\Xi^\#)$ through word-wise comparison:
\[
v \sim w \overset{\mathrm{def}}{\iff} |v| = |w| \text{ and } v_i \sim w_i \text{ for all } 0 \ls i < |v|.
\]
}
We work with three nested languages
\[
\mathcal{L}(\Xi^\mathrm{rec}) \subset \mathcal{L}(\Xi^\#) \subset \mathcal{L}(\Xi).
\]
In general, these three sets are distinct; they coincide if and only if the shift $\Xi$ is the disjoint union of irreducible components.

\begin{definition}[Word Degree, Magic Words and Magic Pairs]
Fix a Bowen relation $\sim$. For any word $w\in \mathcal{L}(\Xi)$ and index $0 \ls i < |w|$, define the \textit{degree of $w$ at position $i$} by
\begin{eqnarray*}
\delta_\sim(w,i) :=  \sharp \{v_i : v \in \mathcal{L}(\Xi^\#),\ v \sim w\}, 
\end{eqnarray*}
and let the degree of $w$ 
\begin{eqnarray*}
\delta_\sim(w) := \min_{0\ls i <|w|} \delta_\sim(w,i).
\end{eqnarray*}The \textit{recurrent degree} of the Bowen relation $\sim$ is the infimum over all recurrent admissible words:
\[
\delta_\mathrm{rec}(\sim) := \inf\Bigl\{ \delta_\sim(w) : w \in \mathcal{L}(\Xi^\mathrm{rec}) \Bigr\}.
\]
A word $w\in \mathcal{L}(\Xi^\mathrm{rec})$ attaining this infimum is called a \textit{magic word}; a pair $(w,i)$ realizing the infimum value is termed a \textit{magic couple}.
\end{definition}

Since $\sim$ is reflexive, one immediately has $\delta_\sim(w,i)\gs 1$ for every recurrent word $w$ and index $i$.
If $\delta_\mathrm{rec}(\sim)<\infty$, magic words and magic couples always exist.
In particular, it holds whenever $\sim$ is locally finite.


\subsubsection{Bowen quotient semi-conjugacy.}

We now present a purely combinatorial construction to produce a new semiconjugacy $\pi_N$ from the original Bowen semiconjugacy $\pi$ (see \cite{buzzi2020} for details). The construction encodes information about the $N$-fold preimage structure of points in the base space $X$.

Fix an integer $N\gs 1$. We first define a directed graph $\mathcal{G}_N$ built from the original transition graph $\mathcal{G}$ of $\Xi$:
\begin{itemize}
    \item Let $\mathcal{A}_N$ be the collection of all $N$-element subsets $A\subset \mathcal{A}$ such that every pair of distinct symbols in $A$ is Bowen equivalent under $\sim$.
    \item Let $\mathcal{G}_N$ be the simple directed graph with vertex set $\mathcal{A}_N$. A directed edge $A\to B$ exists in $\mathcal{G}_N$ if and only if there is a bijection $\phi \colon A \to B$ satisfying: for all $a\in A$ and $ b\in B$, $a\to b$ is an edge in the original graph $\mathcal{G}$ if and only if $b=\phi(a)$.
\end{itemize}
Denote by $(\Xi_N, \sigma_N)$ the symbolic systems generated by $\mathcal{G}_N$ with left shift $\sigma_N$.  
\begin{definition}[$N$-th Order Bowen Quotient]
Let $\pi \colon (\Xi^\#,\sigma) \to (X,T)$ be a semiconjugacy admitting a locally finite Bowen relation $\sim$. For any integer $N\gs 1$, the \textit{$N$-th order Bowen quotient} associated to $(\pi,\sim)$ is the triple
\[
\bigl(\pi_N \colon \Xi_N^\# \to X,\ \stackrel{N}{\sim}\bigr)
\]
defined via the following three conditions:
\begin{enumerate}
    \item $\Xi_N^\#$ is the regular recurrent subset of the Markov shift $\Xi_N$ generated by the transition graph $\mathcal{G}_N$;
    \item The map $\pi_N \colon \Xi_N^\# \to X$ satisfies the invariance rule: for any $\hat{x} \in \Xi_N^\#$ and any bi-infinite lift $x\in \Xi^\#$ with $x_n \in \hat{x}_n$ for all $n\in\mathbb{Z}$, we have $\pi_N(\hat{x}) = \pi(x)$;
    \item The quotient Bowen relation $\stackrel{N}{\sim}$ on $\mathcal{A}_N$ is defined by: $A \stackrel{N}{\sim} B$ if and only if $a\sim b$ for every pair $(a,b)\in A\times B$.
\end{enumerate}
\end{definition}

We next introduce the canonical lift map from quotient sequences back to the original shift space $\Xi$.
\begin{lemma}\label{unique path}
    Take any finite or bi-infinite path $\hat{x} = (\hat{x}_n)_{i<n<j}$ on $\mathcal{G}_N$ with index range $-\infty \ls i<0<j \ls\infty$. For each symbol $a\in \hat{x}_0$, there exists a unique path
\[
Q(\hat{x},a) := (x_n^a)_{i<n<j}
\]
on the original graph $\mathcal{G}$ such that $x_0^a = a$ and $x_n^a \in \hat{x}_n$ holds for all indices $i<n<j$.
\end{lemma}
\begin{proof}
    See Fact 3.9 of \cite{buzzi2020}.
\end{proof}
Fix an arbitrary total ordering on the alphabet $\mathcal{A}$ of $\Xi$. Using the lift $Q$, we define a projection map
\[
q_N \colon \Xi_N \to \Xi,\qquad q_N(\hat{x}) := Q\bigl(\hat{x},\min(\hat{x}_0)\bigr),
\]
where $\min(\hat{x}_0)$ denotes the minimal symbol in the vertex set $\hat{x}_0$ under the chosen total order.
The map $q_N$ is $1$-Lipschitz with respect to the standard shift metric on $\Xi_N$ and $\Xi$, since any central word segment $x_{[-n,n]}$ of $q_N(\hat{x})$ is fully determined by the corresponding segment $\hat{x}_{[-n,n]}$ of $\hat{x}$.
If $\hat{x} \in \Xi_N^\#$, then every lift sequence $x$ satisfying $x_n\in \hat{x}_n$ for all $n\in\mathbb{Z}$ belongs to $\Xi^\#$.

Composing with the original semiconjugacy $\pi$, we obtain a Borel measurable quotient semiconjugacy:
\[
\pi_N \colon (\Xi_N^\#,\sigma_N) \to (X,T),\quad \hat{x} \mapsto \pi\bigl(q_N(\hat{x})\bigr).
\]

We say a bi-infinite quotient sequence $\hat{x} \in \Xi_N^\#$ is \textit{related to a finite word $w\in\mathcal{L}(\Xi)$} if there exist infinitely many integers $p\gs 0$ and infinitely many integers $p \ls 0$ such that
\[
w_i \in \hat{x}_{p+i} \quad \text{for all } 0 \ls i < |w|.
\]
In this situation, we say the point $ \pi_N(\hat{x})$ is related to $w$.

The following definition will be convenient for our study.

\begin{definition}
    A semiconjugacy $\pi:(\Xi,\sigma)\to (X,T)$ is \textit{excellent} for $\sim$ if
\begin{enumerate}
    \item $\sim$ is a locally finite, reflexive and symmetric relation on the alphabet of $\Xi$;
    \item $\pi$ is Bowen with respect to the relation $\sim$;
    \item $\pi$ is Borel and finite-to-one.
\end{enumerate}
\end{definition} 

Throughout this paper, $\binom{p}{q} = \frac{p!}{q!(p-q)!}$ and is zero if
$q>p$.

\begin{theorem}[Theorem 3.3 of Buzzi \cite{buzzi2020}]\label{well-d}Let $\Xi$ be a Markov shift with regular part $\Xi^\#$. Let $\pi:(\Xi^\#,\sigma)\to (X,T)$ be an
excellent semiconjugacy for some Bowen relation $\sim$. Then, for any integer $N\gs 1$,
the Bowen quotient $(\pi_N:\Xi_N^\#\to X,\overset{N}{\sim})$ of order $N$ is well-defined and excellent.
Moreover, the $q_N:\Xi_N\to \Xi$ defined above is a finite-to-one, $1$-Lipschitz map and satisfies:
\begin{enumerate}
    \item if $\Xi$ is locally compact, then so is $\Xi_N$;
    \item {$\pi_N = \pi\circ q_N:\Xi_N^\#\to X$;}
    \item $|\pi_N^{-1}(y)| \ls \binom{|\pi^{-1}(y)|}{N}$ for all $y\in X$ with equality except on a null set;
    \item {$q_N:\Xi_N^\#\to \Xi^\#$ is proper, i.e., for any compact set $K\subset \Xi^\#$, $q_N^{-1}(K)\cap \Xi_N^\#$ is
    compact.}
\end{enumerate}
    
\end{theorem} 
By Theorem \ref{well-d}, the $N$-th order Bowen quotient $(\pi_N \colon \Xi_N^\# \to X,\stackrel{N}{\sim})$ is well-defined  and  excellent for every $N\gs 1$. {Notice that $q_N$ need not commute with the shifts; the semiconjugacy is $\pi_N$, not $q_N$.}

\subsubsection{SPR component}

We now regard $\Xi$ as the countable Markov system $(\Sigma, \sigma)$ from Theorem \ref{the 5} and set $N = \delta_\mathrm{rec}(\sim)$. The shift space $\Sigma$ is generated by the graph $\mathcal{G}$ with vertex set $\mathcal{V}$.

\begin{proposition}
There exists a mixing SPR irreducible component $Z$ of $\Sigma_N$ containing every $\hat{x} \in \Sigma_N^\#$ related to some magic word for $\sim$ over $\mathcal{L}(\Sigma^\mathrm{rec})$.
\end{proposition}

\begin{proof}
By Claim 5.5 of \cite{buzzi2020}, it suffices to verify the SPR property. More precisely, fix a magic couple $(w,i)$ for $\sim$ over $\mathcal{L}(\Sigma^{\mathrm{rec}})$ and define
\[
A_w
\coloneqq
\bigl\{
v_i : v\in\mathcal{L}_{|w|}(\Sigma^\#),\ v\sim w
\bigr\}
\in\mathcal{V}_N.
\]
The proof of Claim 5.5 of \cite{buzzi2020} shows that every $\hat{x}\in\Sigma_N^\#$ related to $w$ visits the symbol $A_w$ infinitely often in both forward and backward time, and hence belongs to the irreducible component $Z_w$ of $\Sigma_N$ containing $A_w$. If $v$ is another magic word, the irreducibility of $\Sigma$ yields a periodic point whose orbit sees both $w$ and $v$ infinitely often. Its lift to $\Sigma_N^\#$ is related to both words, and therefore $Z_w=Z_v$. The resulting irreducible component, independent of the chosen magic word, is denoted by $Z$. {We write $\sigma_Z:=\sigma_N|_Z$ for the restricted shift on $Z$.} Thus $Z$ contains every $\hat{x}\in\Sigma_N^\#$ related to some magic word for $\sim$ over $\mathcal{L}(\Sigma^{\mathrm{rec}})$. Moreover, Claim 5.5 of \cite{buzzi2020} gives
\[
\delta_{\mathrm{rec}}\bigl(\pi_N|_{Z^\#}\bigr)=1.
\]
It remains to prove that $Z$ is SPR.

{
Let $w$ be a magic word of $\Sigma$. Denote by $\mu$ the unique MME on $\mathcal{H}$. By Theorem \ref{the 5}, $\mu$ admits an entropy-preserving invariant lift $\mu_1$ to $\Sigma$. This lift is the unique MME on $\Sigma$ by \cite{Buzzi-Sarig}, and it has full support since $\Sigma$ is irreducible. Thus
\[
\pi_*\mu_1=\mu.
\]
Since $\mu_1([w])>0$, almost every orbit with respect to $\mu_1$ visits $[w]$ infinitely often. By Theorem 4.6 of \cite{buzzi2020} and the Bowen property,
\[
\sharp\bigl(\pi^{-1}(y)\cap\Sigma^\#\bigr)=N
\]
for $\mu$-almost every $y$. Applying Theorem 5.2 of \cite{buzzi2020} to the $N$-th order Bowen quotient gives an ergodic invariant measure $\mu_2$ on $\Sigma_N$ such that
\[
(\pi_N)_*\mu_2=\mu
\]
and $\pi_N:(\Sigma_N,\mu_2,\sigma_N)\to(\mathcal H,\mu,f)$ is a measure-theoretic isomorphism. For $\mu$-almost every $y$, one can choose $x\in\pi^{-1}(y)\cap\Sigma^\#$ whose orbit sees $w$ infinitely often. Since this fiber has cardinality $N$, the unique point $\hat{x}\in\pi_N^{-1}(y)$ contains $x$ coordinatewise and is therefore related to $w$. Claim 5.5 of \cite{buzzi2020} then implies that $\mu_2$ is carried by $Z$. {In what follows, we write $\mu_Z:=\mu_2$, viewed as a measure on $Z$.} Consequently,
{
\[
h_{\mathrm{top}}(\sigma_Z,Z)
\gs
h_{\mu_Z}(\sigma_Z)
=
h_\mu(f)
=
h_{\mathrm{top}}(f,\mathcal{H}).
\]
}
}
{
In addition, for any ergodic invariant measure
\[
\rho\in\mathcal M_{\erg}(\sigma_Z,Z),
\]
the pushforward
\[
(\pi_N)_*\rho
\coloneqq
\rho\circ\pi_N^{-1}
\]
is an ergodic $f$-invariant measure on $\mathcal H$, as $\pi_N$ defines a semiconjugacy from $(Z,\sigma_Z)$ to $(\mathcal H,f)$. Since $\pi_N$ is finite-to-one, metric entropy is preserved:
\[
h_{(\pi_N)_*\rho}(f)
=
h_\rho(\sigma_Z).
\]
}
Since $\rho(Z^\#)=1$, consider the finite extension
\[
\widetilde Z:=\bigl\{(\hat{x},x)\in Z^\#\times\Sigma^\#:x_n\in\hat{x}_n\text{ for all }n\in\mathbb Z\bigr\},
\qquad
\widetilde\sigma(\hat{x},x):=({\sigma_Z}\hat{x},\sigma x).
\]
For each $\hat{x}$, its fiber consists of the $N$ paths $Q(\hat{x},a)$, $a\in\hat{x}_0$. The uniform measures on these fibers are equivariant and define a $\widetilde\sigma$-invariant lift $\widetilde\rho$ of $\rho$. Let $\widehat\rho$ be its projection to $\Sigma$. Then $\widehat\rho$ is $\sigma$-invariant and
\[
\pi_*\widehat\rho=(\pi_N)_*\rho.
\]
Theorem \ref{the 5} therefore implies that $(\pi_N)_*\rho$ is $\chi'$-hyperbolic and
\[
(\pi_N)_*\rho\bigl(\mathrm{NUH}_{\chi'}^\#\bigr)=1.
\]

It then follows that
\[
{
h_{\mathrm{top}}(\sigma_Z,Z)
=
\sup
\bigl\{
h_\rho(\sigma_Z):
\rho\in\mathcal M_{\erg}(\sigma_Z,Z)
\bigr\}
\ls
h_{\mathrm{top}}(f,\mathcal H).
}
\]
Combining the two entropy inequalities yields
\[
{
h_{\mathrm{top}}(\sigma_Z,Z)
=
h_{\mathrm{top}}(f,\mathcal H).
}
\]

Moreover, $Z$ is topologically mixing. Indeed, in the present reduction
$l(\mathcal H)=1$, and \cite[Theorem 2]{BCS} implies that
$(\mathcal H,\mu,f)$ is Bernoulli, hence mixing. Since
\[
\pi_N:(Z,\mu_Z,\sigma_Z)\longrightarrow(\mathcal H,\mu,f)
\]
is a measure-theoretic isomorphism, $\mu_Z$ is mixing. The measure $\mu_Z$
is the MME of the irreducible Markov shift $Z$ and has full support.
Therefore the period of $Z$ is equal to $1$, and hence $(Z,\sigma_Z)$ is
topologically mixing.

{By Theorem \ref{spr h x}, for the fixed rate $\chi_0>0$ and all sufficiently small $\vep>0$, one can find $\tau\in(0,1)$, a constant $h_0<h_{\mathrm{top}}(f,\mathcal H)$, and a $(\chi_0,\vep)$-Pesin block $B$ such that}
\begin{eqnarray}\label{spr h}
\forall\,\nu\in \mathcal{M}_{\erg}(f, \mathcal{H}):\quad h_\nu(f) > h_0 \Rightarrow \nu(B) > \tau,
\end{eqnarray}
as in Definition \ref{SPR}. Observe that any $(\chi_0, \vep)$-Pesin block is also a $(\chi, \vep)$-Pesin block whenever $0 < \chi < \chi_0$.

Concerning recurrence properties, for each $\chi' > 0$, \cite[Section 3.3.1]{BCS} introduces a Borel subset $\mathrm{NUH}^{\#}_{\chi'} \subseteq \mathcal{H}$ of full measure for every saddle-type $\chi'$-hyperbolic measure supported on $\mathcal{H}$. Choose any $0 < \chi' < \chi - \vep$ as in Theorem \ref{the 5}; then every $(\chi, \vep)$-Pesin block is contained in $\mathrm{NUH}_{\chi'}(f)$.

\begin{lemma}[(P3) of \cite{BCS} (Page 453), Lemma 10.3 of \cite{SPR2025}]\label{belong}
For the Pesin block $B$, there exist vertices $V_1, \dots, V_{m} \in \mathcal{V}$ such that
\[
\bigl(\pi^{-1}\bigl(B \cap \mathrm{NUH}_{\chi'}^{\#}\bigr)\bigr) \cap \Sigma^{\#} \subset \bigcup_{1 \ls j \ls m} [V_j].
\]
\end{lemma}

Recall the factorization $\pi_N = \pi \circ q_N$. Applying Lemma \ref{belong},
\[
q_N\bigl(\pi_N^{-1}\bigl(B \cap \mathrm{NUH}_{\chi'}^{\#}\bigr)\bigr) \cap \Sigma^{\#} \subseteq \pi^{-1}\bigl(B \cap \mathrm{NUH}_{\chi'}^{\#}\bigr) \cap \Sigma^{\#} \subseteq \bigcup_{1 \ls j \ls m} [V_j].
\]

For each index $1 \ls j \ls m$, define the finite collection
\[
A_j = \bigl\{A\in \mathcal{V}_N: V \sim V_j \text{ for all } V\in A\bigr\},
\]
which is finite by local finiteness of the Bowen relation $\sim$. Note that $q_N(\Sigma_N^\#) \subseteq \Sigma^\#$. We then deduce
\[
{\pi_N^{-1}\bigl(B \cap \mathrm{NUH}_{\chi'}^{\#}\bigr) \cap Z^{\#} \subseteq \bigcup_{A\in \bigcup_{1 \ls j \ls m} A_j} [A]_Z \eqqcolon D,}
\]
where $[A]_Z$ denotes the cylinder set of the symbol $A$ within the component $Z$.

{
Take an arbitrary
\[
\rho\in\mathcal M_{\erg}(\sigma_Z,Z)
\]
with
\[
h_\rho(\sigma_Z)>h_0.
\]
Its pushforward satisfies
\[
(\pi_N)_*\rho\in\mathcal M_{\erg}(f,\mathcal H)
\]
and
\[
h_{(\pi_N)_*\rho}(f)
=
h_\rho(\sigma_Z)
>
h_0.
\]
}
Using the SPR estimate \eqref{spr h}, we compute
\begin{eqnarray*}
\rho(D) &\gs& \rho\Bigl(\pi_N^{-1}\bigl(B \cap \mathrm{NUH}_{\chi'}^{\#}\bigr) \cap {
Z^{\#}}\Bigr)\\[2mm]
&=& (\pi_N)_*\rho\bigl(B \cap \mathrm{NUH}_{\chi'}^{\#}\bigr)
= (\pi_N)_*\rho(B) > \tau.
\end{eqnarray*}
By assertion (2) of Proposition \ref{spr e}, this verifies that $Z$ satisfies the SPR condition.
\end{proof}

{Applying the argument of Buzzi \cite[Theorem 1.2]{buzzi2020} to the irreducible component associated with $\mathcal H$, we obtain the following lower bound:}
\[
\liminf_{ n \to \infty} \frac{\sharp P_{\chi_1,\chi_2}(f, \mathcal{H}, n)}{e^{n h_{\mathrm{top}}(f, \mathcal{H})}} \gs 1.
\]
Combined with Proposition \ref{away}, it remains to prove the corresponding upper bound:
\[
\limsup_{n \to \infty} \frac{\sharp P^c_{\chi_1,\chi_2}(f, \mathcal{H}, \eta, n)}{e^{n h_{\mathrm{top}}(f, \mathcal{H})}} \ls 1.
\]

\subsection{Continuity of Lyapunov exponents}\label{sec 4.3}

For any Lyapunov regular point $x$, let $\lambda_1(x,f) \gs \lambda_2(x,f)$ denote its two Lyapunov exponents. For each $f$-invariant Borel probability measure $\nu$, define the integrated Lyapunov exponents by
\[
\lambda_1(\nu,f)
=
\int \lambda_1(x,f)\,\mathrm{d}\nu,
\qquad
\lambda_2(\nu,f)
=
\int \lambda_2(x,f)\,\mathrm{d}\nu.
\]
Recall the trace identity
\[
\lambda_1(\nu,f)+\lambda_2(\nu,f)
=
\int \log |\det(D_xf)|\,\mathrm{d}\nu,
\]
which implies that the sum $\lambda_1+\lambda_2$ is continuous with respect to the weak$^*$ topology on the space of invariant measures. Together with the lower semicontinuity of $\lambda_2$, this implies that, for every sufficiently small
\[
0<\delta<\frac{1}{2}\min\{t_1,-t_2\},
\]
there exists $\eta_0=\eta_0(\delta)>0$ such that every ergodic $f$-invariant measure $\nu$ satisfying $d_{\mathcal{M}}(\nu,\mu)<\eta_0$ also satisfies
\[
\bigl|\lambda_1(\nu,f)+\lambda_2(\nu,f)-t_1-t_2\bigr|<\delta,
\qquad
\lambda_2(\nu,f)>t_2-\delta.
\]
If, in addition, $\lambda_2(\nu,f)<t_2+\delta$, then
\[
|\lambda_2(\nu,f)-t_2|<\delta
\qquad\text{and}\qquad
|\lambda_1(\nu,f)-t_1|<2\delta.
\]
Consequently, under these assumptions, $\lambda_1(\nu,f)$ and $\lambda_2(\nu,f)$ approximate $t_1$ and $t_2$, respectively.

\begin{theorem}[\cite{SPR2025}, Proposition 2.11 of \cite{LTTY2026}]
\label{large measure}
There exist $\chi>0$ such that for any $b\in(0,1)$, there exists $\delta>0$ such that, for all sufficiently small $\varepsilon>0$, one can construct a $(\chi,\varepsilon)$-Pesin block $\Lambda$ and a constant $\eta_1>0$ with the following property: every ergodic $f$-invariant measure $\nu$ such that
\[
d_{\mathcal{M}}(\nu,\mu)<\eta_1
\qquad\text{and}\qquad
\lambda_2(\nu,f)<t_2+\delta
\]
satisfies
\[
\nu(\Lambda)>1-b.
\]
\end{theorem}

We now fix the parameters $\chi_1$ and $\chi_2$ appearing in our main theorems. Let
\[
\chi_1\in\bigl(0,h_{\top}(f,\H)\bigr)
\]
be arbitrary and fixed in the rest of the paper.
{Let $\chi_{\mathrm L}>0$ be the constant in Theorem \ref{large measure}. Replacing the hyperbolicity rate by a smaller one if necessary, fix
\[
0<\chi<\frac{\min\{\chi_{\mathrm L},\chi_0,\chi_1,t_1,|t_2|\}}{100}.
\]
The conclusion of Theorem \ref{large measure} is unchanged, since every $(\chi_{\mathrm L},\varepsilon)$-Pesin block is also a $(\chi,\varepsilon)$-Pesin block. This $\chi$ is also used in the coding Theorem \ref{the 5}.}

Then, we define $\chi_2$. {Set $b:=\frac{1}{4}$ and apply Theorem \ref{large measure} with this fixed value of $b$.}
This gives the constant $\delta>0$.
By decreasing $\delta$ if necessary, we may also assume that
\[
0<\delta<\frac{1}{2}\min\{t_1,-t_2\}.
\]
Fix $\varepsilon>0$ sufficiently small so that Theorem \ref{large measure} applies, and $\varepsilon<\chi$, and let $\Lambda$ and $\eta_1$ be the corresponding $(\chi,\varepsilon)$-Pesin block and constant. For this choice of $\delta$, choose $\eta_0>0$ as in the continuity argument above, and decrease $\eta_0$ if necessary so that
\(
\eta_0\leqslant \eta_1.
\)
We then define
\begin{equation}\label{chi1}
\chi_2=t_2+\delta.
\end{equation}
{It follows from Theorem \ref{large measure} that, for this fixed constant $b=\frac{1}{4}$ and every}
\[
x\in\bigcup_n P^c_{\chi_1,\chi_2}
(f,\mathcal{H},\eta_0,n),
\]
the periodic measure $\mu_{\mathcal{O}(x)}$ satisfies
\begin{equation}\label{b-large}
\mu_{\mathcal{O}(x)}(\Lambda)>1-b.
\end{equation}

\begin{proposition}\label{homoclinic}
{For the fixed constant $b=\frac{1}{4}$, there exists
$\eta\in(0,\eta_0)$ such that}
\[
\bigcup_n
P^c_{\chi_1,\chi_2}
(f,\mathcal{H},\eta,n)
\subseteq \mathcal{X}.
\]
\end{proposition}

\begin{proof}
On the Pesin block $\Lambda$, the local stable and unstable manifolds have uniformly positive sizes and uniform transverse-intersection properties. Accordingly, there exists $\gamma_0>0$ such that, for any pair $y,z\in\Lambda$ with $d(y,z)<\gamma_0$,
\[
W^s_{\mathrm{loc}}(z)
\text{ intersects }
W^u_{\mathrm{loc}}(y)
\text{ transversally},
\qquad
W^u_{\mathrm{loc}}(z)
\text{ intersects }
W^s_{\mathrm{loc}}(y)
\text{ transversally}.
\]
In particular, any such points $y$ and $z$ are homoclinically related.

Choose $0<\gamma<\gamma_0/2$, and let $U$ be the open $\gamma$-neighborhood of
\[
\supp(\mu\mid_{\Lambda}).
\]
{Since $\mu$ itself satisfies the hypotheses of Theorem \ref{large measure} with the fixed value $b=\frac{1}{4}$, we have}
\[
\mu(\Lambda)>1-b.
\]
Moreover,
\[
\mu(U)\geqslant\mu(\Lambda)>1-b.
\]
Because $U$ is open and $d_{\mathcal{M}}$ is compatible with the weak$^*$ topology, after choosing $\eta\in(0,\eta_0)$ sufficiently small, every invariant probability measure $\nu$ satisfying
\[
d_{\mathcal{M}}(\nu,\mu)<\eta
\]
also satisfies
\[
\nu(U)>1-2b.
\]

Take an arbitrary point
\[
x\in
\bigcup_n
P^c_{\chi_1,\chi_2}
(f,\mathcal{H},\eta,n).
\]
Then
\[
d_{\mathcal{M}}
\bigl(\mu_{\mathcal{O}(x)},\mu\bigr)<\eta,
\]
and hence
\[
\mu_{\mathcal{O}(x)}(U)>1-2b.
\]
Since $\eta<\eta_0$, estimate \eqref{b-large} also gives
\[
\mu_{\mathcal{O}(x)}(\Lambda)>1-b.
\]
Combining these two estimates, we obtain
\[
\mu_{\mathcal{O}(x)}(U\cap\Lambda)
\geqslant
\mu_{\mathcal{O}(x)}(U)
+
\mu_{\mathcal{O}(x)}(\Lambda)-1
>
1-3b>0.
\]

Choose a point
\[
z\in\mathcal{O}(x)\cap U\cap\Lambda.
\]
By the definition of $U$, there exists
\(
y\in\supp(\mu\mid_{\Lambda})
\)
such that
\[
d(y,z)<\gamma.
\]
For every
\[
\xi\in B(y,\gamma)\cap\Lambda,
\]
we have
\[
d(z,\xi)
\leqslant d(z,y)+d(y,\xi)
<2\gamma<\gamma_0.
\]
It follows from the uniform transverse-intersection property of $\Lambda$ that $z$ is homoclinically related to every
\[
\xi\in B(y,\gamma)\cap\Lambda.
\]
Since
\(
y\in\supp(\mu\mid_{\Lambda}),
\)
we have
\[
\mu\bigl(B(y,\gamma)\cap\Lambda\bigr)>0.
\]
Thus the periodic measure $\mu_{\mathcal{O}(x)}$ is homoclinically related to $\mu$. Since $\mathcal{X}$ is the Borel homoclinic class carrying $\mu$, this implies
\[
\mathcal{O}(x)\subseteq\mathcal{X},
\]
and, in particular, $x\in\mathcal{X}$. This proves Proposition \ref{homoclinic}.
\end{proof}

Proposition \ref{homoclinic} guarantees that all points in
\[
\bigcup_n
P^c_{\chi_1,\chi_2}
(f,\mathcal{H},\eta,n)
\]
admit symbolic codings in the shift space $\Sigma$.

Parallel to Lemma \ref{belong}, for the Pesin block $\Lambda$, there exist vertices
\[
R_1,\dots,R_{\kappa}\in\mathcal{V}
\]
such that
\begin{align}\label{finite cy}
\pi^{-1}
\bigl(\Lambda\cap\mathrm{NUH}_{\chi'}^{\#}\bigr)
\cap\Sigma^{\#}
\subset
\bigcup_{1\ls j\ls \kappa}[R_j].
\end{align}

For any
\[
x\in
\left(
\bigcup_n
P^c_{\chi_1,\chi_2}
(f,\mathcal{H},\eta,n)
\right)
\cap\Lambda
\subset
\Lambda\cap\mathrm{NUH}_{\chi'}^{\#},
\]
one has
\[
\pi^{-1}(x)\cap\Sigma^{\#}
\subseteq
\bigcup_{1\ls j\ls \kappa}[R_j].
\]
Moreover, {for every lift $z\in\pi^{-1}(x)\cap\Sigma^\#$}, the equivariance relation
\[
\pi\circ\sigma=f\circ\pi
\]
yields
\[
\pi\bigl(\sigma^{n_0}(z)\bigr)
=
f^{n_0}\bigl(\pi(z)\bigr)
=
f^{n_0}(x)
=
x,
\] 
where $n_0$ is the period of $x$. Therefore,
\[
{\sigma^{n_0}(z)\in\pi^{-1}(x)\cap\Sigma^\#,}
\]
and hence
\[
\sigma^{n_0}(z)
\in
\bigcup_{1\ls j\ls \kappa}[R_j].
\]

\subsection{Periodic orbits avoiding the magic cylinder}

Let \(w\) be a magic word of \(\Sigma\). In this section, we investigate periodic orbits of \(f\) whose symbolic lifts on \(\Sigma\) never encounter the magic word \(w\) and prove that their exponential growth rate is strictly less than the topological entropy. Define 
\[
P^w_n=\left\{x\in P^c_{\chi_1,\chi_2}
(f,\mathcal{H},\eta,n):
\sigma^j\bigl(\pi^{-1}(x)\cap\Sigma^\#\bigr)\cap[w]=\emptyset,\quad
0\ls j\ls n-1\right\}.
\]

For arbitrary words \(W_1,W_2,W_3\in\mathcal{L}(\Sigma)\), let
\([W_1],[W_2],[W_3]\) be the corresponding cylinder sets in \(\Sigma\).
For $n>\max\{|W_1|,|W_2|,|W_3|\}$, we define
\begin{align*}
\mathcal{C}_n(W_1,W_2;W_3)
=\bigl\{W\in\mathcal{L}_n(\Sigma):\,\,
&[W]\subseteq[W_1],\quad
\sigma^{n-|W_2|}\bigl([W]\bigr)\cap[W_2]\neq\emptyset,\\
&\sigma^j\bigl([W]\bigr)\cap[W_3]=\emptyset,\quad
\forall\,1\ls j<n-|W_3|
\bigr\}.
\end{align*}

\begin{lemma}\label{through}
Let \(V_1,V_2\) be any two vertices of the graph \(\mathcal{G}\). Then
\[
\limsup_{n\to\infty}\frac{1}{n}
\log\sharp\mathcal{C}_n(V_1,V_2;w)
\ls
\limsup_{n\to\infty}\frac{1}{n}
\log\sharp\mathcal{C}_n(w,w;w)
<
h_{\top}(\sigma,\Sigma).
\]
\end{lemma}

\begin{proof}
Set \(m=|w|\). Consider the countable Markov shift
\((\Sigma_1,\sigma_1)\) obtained by taking all words of length \(m\) in
\(\Sigma\) as the vertices of the new state graph. Then
\((\Sigma_1,\sigma_1)\) is topologically conjugate to
\((\Sigma,\sigma)\). Recall that the SPR property is invariant under
topological conjugacy; see assertion~(3) of Proposition~\ref{spr e}.

{In the higher-block shift,
\[
\sharp\mathcal{C}_{\Sigma_1,n}(w,w;w)
=
\mathrm{F}^*_{n-1}([w],0,\sigma_1),
\]
and this shift of one in the length does not affect the exponential growth rate.}

By the SPR condition, we obtain the strict entropy bound
\[
\limsup_{n\to\infty}\frac{1}{n}
\log\sharp\mathcal{C}_{\Sigma_1,n}(w,w;w)
<
h_{\top}(\sigma_1,\Sigma_1)
=
h_{\top}(\sigma,\Sigma),
\]
where \(\mathcal{C}_{\Sigma_1,n}\) is defined analogously to
\(\mathcal{C}_n\) using the language \(\mathcal{L}_n(\Sigma_1)\).
Each word of length \(n\) in \(\Sigma_1\) corresponds canonically to a
word of length \(n+m-1\) in \(\Sigma\).

Now fix arbitrary vertices \(V_1,V_2\in\mathcal{V}\). By the local
finiteness of \(\mathcal{G}\), the families
\[
\mathcal{U}_1
=
\left\{
W_1\in\mathcal{L}_m(\Sigma):
[W_1]\subseteq[V_1]
\right\}
\]
and
\[
\mathcal{U}_2
=
\left\{
W_2\in\mathcal{L}_m(\Sigma):
\sigma^{m-1}([W_2])\subseteq[V_2]
\right\}
\]
are finite. Thus, \(\mathcal{U}_1\) consists of the length-\(m\) words
beginning at \(V_1\), whereas \(\mathcal{U}_2\) consists of the
length-\(m\) words ending at \(V_2\).

By topological transitivity, for every
\(W_1\in\mathcal{U}_1\) and \(W_2\in\mathcal{U}_2\), there exist a
connecting path of length \(l(w,W_1)\gs0\) from \(w\) to \(W_1\) and a
return path of length \(l(W_2,w)\gs0\) from \(W_2\) to \(w\), such that
neither path passes through \(w\) at an intermediate vertex. Connecting
these fixed paths gives an injection from $\sharp\mathcal{C}_{\Sigma_1,n}(W_1,W_2;w)$ to $\sharp\mathcal{C}_{\Sigma_1,
n+l(w,W_1)+l(W_2,w)}(w,w;w)$, and hence
\[
\sharp\mathcal{C}_{\Sigma_1,n}(W_1,W_2;w)
\ls
\sharp\mathcal{C}_{\Sigma_1,
n+l(w,W_1)+l(W_2,w)}(w,w;w).
\]

The canonical higher-block identification gives
\[
\sharp\mathcal{C}_{\Sigma_1,n}(W_1,W_2;w)
=
\sharp\mathcal{C}_{n+m-1}(W_1,W_2;w).
\]
Consequently,
\begin{align*}
&\limsup_{n\to\infty}\frac{1}{n}
\log\sharp\mathcal{C}_n(V_1,V_2;w)\\[2mm]
&\ls
\limsup_{n\to\infty}\frac{1}{n}
\log
\sum_{\substack{W_1\in\mathcal{U}_1\\W_2\in\mathcal{U}_2}}
\sharp\mathcal{C}_{\Sigma_1,n-m+1}(W_1,W_2;w)\\[2mm]
&\ls
\limsup_{n\to\infty}\frac{1}{n}
\log
\sum_{\substack{W_1\in\mathcal{U}_1\\W_2\in\mathcal{U}_2}}
\sharp\mathcal{C}_{\Sigma_1,
n-m+1+l(w,W_1)+l(W_2,w)}(w,w;w)\\[2mm]
&\ls
\limsup_{n\to\infty}\frac{1}{n}
\log\sharp\mathcal{C}_{\Sigma_1,n}(w,w;w)\\[2mm]
&<
h_{\top}(\sigma,\Sigma).
\end{align*}
Here the third inequality follows because
\(\mathcal{U}_1\) and \(\mathcal{U}_2\) are finite and all the changes in
word length are bounded independently of \(n\).

Finally, the canonical identification
\[
\sharp\mathcal{C}_{\Sigma_1,n}(w,w;w)
=
\sharp\mathcal{C}_{n+m-1}(w,w;w)
\]
implies
\[
\limsup_{n\to\infty}\frac{1}{n}
\log\sharp\mathcal{C}_{\Sigma_1,n}(w,w;w)
=
\limsup_{n\to\infty}\frac{1}{n}
\log\sharp\mathcal{C}_n(w,w;w).
\]
This proves the lemma.
\end{proof}

\begin{proposition}\label{not see magic}
    \[
\limsup_{n\to \infty} \frac{1}{n} \log\sharp P^w_n <h_{\top}(\sigma,\Sigma).
\]
\end{proposition}

\begin{proof}
We first establish a uniform upper bound for the growth rate of admissible symbolic cylinder sets, which serves as the key quantitative ingredient for the final entropy estimate.

Take arbitrary vertex states $R_i, R_j\in \mathcal{V}$ appearing in \eqref{finite cy} with $1\ls i,j\ls \kappa$. By Lemma \ref{through}, the cardinality of finite admissible orbit cylinders satisfies
\begin{align}
\limsup_{n\to\infty}\frac{1}{n}
\log\sharp\mathcal{C}_n(R_i,R_j;w)
&\ls
\limsup_{n\to\infty}\frac{1}{n}
\log\sharp\mathcal{C}_n(w,w;w) \notag\\[2mm]
&
<
h_{\top}(\sigma,\Sigma).
\label{small minset}
\end{align}

Define the critical growth constant
\[
c_0:=\limsup_{n\to \infty} \frac{1}{n} \log\sharp  \mathcal{C}_{n}(w, w; w),
\]
so we have $c_0< h_{\top}(\sigma,\Sigma)$.

The core local packing estimate is provided by the following lemma, which controls the number of periodic points inside any dynamical ball of fixed scale.

\begin{lemma}\label{small loc}
For any $\rho>0$, there exist $\delta>0$ sufficiently small such that for $\chi_1$, $\chi_2$, and $\eta$ defined as in \Cref{sec 4.3}, all sufficiently large $n\in\mathbb N$ and every orbit point $x\in P^c_{\chi_1,\chi_2}
(f,\mathcal{H},\eta,n)$,
\[
 \sharp \bigl(P^c_{\chi_1,\chi_2}
(f,\mathcal{H},\eta,n) \cap B_n(x, \delta)\bigr)<e^{n\rho}.
\]
\end{lemma}

\begin{proof}[Proof of Lemma \ref{small loc}]
Set
\[
\chi_*:=\min\{\chi_1,-\chi_2\}>0,
\]
and let $\mathcal{P}_{\chi_*}$ denote the set of saddle periodic
points of $f$ whose Lyapunov exponents lie outside the interval
$[-\chi_*,\chi_*]$. For every $n\in\mathbb{N}$, set
\[
\mathcal{P}_{\chi_*,n}
:=
\bigl\{
p\in\mathcal{P}_{\chi_*}:f^n(p)=p
\bigr\}.
\]

By definition, it follows that
\[
P^c_{\chi_1,\chi_2}
(f,\mathcal{H},\eta,n)
\subseteq
\mathcal{P}_{\chi_*,n}
\]
for every $n\in\mathbb{N}$.
By the asymptotically periodic expansive property established by
Burguet \cite{Bur},
\[
g^*_{\mathcal{P}_{\chi_*}}=0.
\]
Hence, for the given $\rho>0$, there exists $\delta>0$ sufficiently
small such that
\[
g_{\mathcal{P}_{\chi_*}}(\delta)
=
\limsup_{n\to\infty}
\frac{1}{n}
\sup_{y\in M}
\log
\sharp\bigl(
\mathcal{P}_{\chi_*,n}\cap B_n(y,\delta)
\bigr)
<
\frac{\rho}{2}.
\]
It follows that there
exists $n_0\in\mathbb{N}$ such that, for every $n\gs n_0$,
\[
\frac{1}{n}
\sup_{y\in M}
\log
\sharp\bigl(
\mathcal{P}_{\chi_*,n}\cap B_n(y,\delta)
\bigr)
<
\rho.
\]
Equivalently, for every $n\gs n_0$,
\[
\sup_{y\in M}
\sharp\bigl(
\mathcal{P}_{\chi_*,n}\cap B_n(y,\delta)
\bigr)
<
e^{n\rho}.
\]
Now fix $n\gs n_0$ and
\(
x\in
P^c_{\chi_1,\chi_2}
(f,\mathcal{H},\eta,n).
\)
Using the inclusion established above, we obtain
\[
\begin{aligned}
\sharp\Bigl(
P^c_{\chi_1,\chi_2}
(f,\mathcal{H},\eta,n)
\cap B_n(x,\delta)
\Bigr)
&\ls
\sharp\bigl(
\mathcal{P}_{\chi_*,n}
\cap B_n(x,\delta)
\bigr)\\[2mm]
&\ls
\sup_{y\in M}
\sharp\bigl(
\mathcal{P}_{\chi_*,n}
\cap B_n(y,\delta)
\bigr)\\[2mm]
&<
e^{n\rho}.
\end{aligned}
\]
This proves the lemma.
\end{proof}

We now complete the main proposition via separated set counting and symbolic coding entropy comparison.

Fix
\[
0<\rho<h_{\top}(\sigma,\Sigma)-c_0,
\]
and let $\delta>0$ be given by Lemma \ref{small loc}. For every sufficiently large $n$, let
\[
S_n\subset P_n^w
\]
be a maximal $(n,\delta)$-separated subset. Then
\[
P_n^w\subseteq \bigcup_{x\in S_n} B_n(x,\delta).
\]

By \eqref{b-large}, the periodic orbit of every $x\in S_n$ intersects the Pesin block $\Lambda$. For each $x\in S_n$, choose an integer
\[
k(x)\in\{0,\dots,n-1\}
\]
such that
\[
f^{k(x)}(x)\in\Lambda.
\]
For each $0\ls k\ls n-1$, define
\[
S_{n,k}:=\{x\in S_n:k(x)=k\},
\qquad
Y_{n,k}:=f^k(S_{n,k}).
\]
Writing
\[
d_n(x,y):=\max_{0\ls r<n}d(f^r(x),f^r(y)),
\]
the $n$-periodicity of all points under consideration gives
\[
d_n(f^k(x),f^k(y))=d_n(x,y).
\]
Hence, every $Y_{n,k}$ is an $(n,\delta)$-separated subset of $\Lambda$.

By definition, it follows that
\[
Y_{n,k}\subseteq P_n^w\cap\Lambda.
\]

By Proposition \ref{homoclinic} and Theorem \ref{the 5}, every
$y\in Y_{n,k}$ admits a regular symbolic lift. Choose one point
\[
z_y\in\pi^{-1}(y)\cap\Sigma^\#
\]
for each $y\in Y_{n,k}$, and set
\[
\widetilde{Y}_{n,k}:=\{z_y:y\in Y_{n,k}\}.
\]
Then
\[
\sharp\widetilde{Y}_{n,k}
=
\sharp Y_{n,k}
=
\sharp S_{n,k}.
\]

Since $\pi$ is H\"older continuous, there exist constants $C>0$ and
$\beta\in(0,1)$ such that
\[
d(\pi(u),\pi(v))
\ls
C\,d_\Sigma(u,v)^\beta
\]
for all $u,v\in\Sigma$. Define
\[
\delta_1:=(C^{-1}\delta)^{1/\beta}.
\]
If $z_y,z_{y'}\in\widetilde{Y}_{n,k}$ are distinct, then the
$(n,\delta)$-separation of $Y_{n,k}$ gives an integer
$0\ls r<n$ such that
\[
\delta
\ls
d(f^r(y),f^r(y'))
=
d\bigl(\pi(\sigma^r(z_y)),\pi(\sigma^r(z_{y'}))\bigr).
\]
It follows that
\[
d_\Sigma(\sigma^r(z_y),\sigma^r(z_{y'}))
\gs
\delta_1.
\]
Thus $\widetilde{Y}_{n,k}$ is an $(n,\delta_1)$-separated subset of
$\Sigma$.

Since $y\in\Lambda\cap\mathrm{NUH}_{\chi'}^\#$, relation
\eqref{finite cy} implies that
\[
z_y\in[R_i]
\]
for some $1\ls i\ls \kappa$. Furthermore,
\[
\pi(\sigma^n(z_y))
=
f^n(\pi(z_y))
=
f^n(y)
=
y,
\]
and $\sigma^n(z_y)\in\Sigma^\#$. Hence
\[
\sigma^n(z_y)\in[R_j]
\]
for some $1\ls j\ls \kappa$. Since $y\in P_n^w$, the length-$(n+1)$ word
\(
(z_y)_{[0,n]}
\)
does not encounter $w$ and therefore belongs to
\(
\mathcal{C}_{n+1}(R_i,R_j;w).
\)

Fix $N_1\in\mathbb{N}$ sufficiently large so that
\(
e^{-N_1}<\delta_1.
\)
Because the graph $\mathcal{G}$ is locally finite, the number of
admissible extensions of a word
\(
W\in\mathcal{C}_{n+1}(R_i,R_j;w)
\)
from the coordinate interval $[0,n]$ to
\(
[-N_1,n+N_1]
\)
is finite. Since the initial and terminal vertices range over the
finite set $\{R_1,\dots,R_\kappa\}$, there exists a constant $C_1>0$,
independent of $n$, $i$, $j$, and $W$, such that every such word has
at most $C_1$ admissible extensions.

If two points $z,z'\in\widetilde{Y}_{n,k}$ determine the same extended
block on $[-N_1,n+N_1]$, then, for every $0\ls r<n$, the sequences
$\sigma^r(z)$ and $\sigma^r(z')$ agree on all coordinates
$[-N_1,N_1]$. Therefore,
\[
d_\Sigma(\sigma^r(z),\sigma^r(z'))
\ls
e^{-(N_1+1)}
<
\delta_1,
\]
contradicting the $(n,\delta_1)$-separation of
$\widetilde{Y}_{n,k}$. Consequently, each word in
$\mathcal{C}_{n+1}(R_i,R_j;w)$ corresponds to at most $C_1$ points of
$\widetilde{Y}_{n,k}$. Hence
\[
\sharp S_{n,k}
=
\sharp\widetilde{Y}_{n,k}
\ls
C_1
\sum_{1\ls i,j\ls \kappa}
\sharp\mathcal{C}_{n+1}(R_i,R_j;w).
\]
Summing over $0\ls k\ls n-1$, we obtain
\[
\sharp S_n
\ls
nC_1
\sum_{1\ls i,j\ls \kappa}
\sharp\mathcal{C}_{n+1}(R_i,R_j;w).
\]

Finally, the covering property of $S_n$ and Lemma \ref{small loc}
give
\[
\begin{aligned}
\sharp P_n^w
&\ls
\sum_{x\in S_n}
\sharp\bigl(P_n^w\cap B_n(x,\delta)\bigr)\\[2mm]
&\ls
\sum_{x\in S_n}
\sharp\Bigl(
P^c_{\chi_1,\chi_2}(f,\mathcal{H},\eta,n)
\cap B_n(x,\delta)
\Bigr)\\[2mm]
&<
e^{n\rho}\sharp S_n.
\end{aligned}
\]
Therefore,
\begin{eqnarray*}
\limsup_{n\to\infty}
\frac{1}{n}\log\sharp P_n^w
&\ls&
 \limsup_{n\to\infty}
\frac{1}{n}
\log\left(
nC_1
\sum_{1\ls i,j\ls \kappa}
\sharp\mathcal{C}_{n+1}(R_i,R_j;w)
\right)
+\rho\\[2mm]
&\ls&
c_0+\rho.
\end{eqnarray*}
Here the last inequality follows from Lemma \ref{through}, the
finiteness of the set of pairs $(R_i,R_j)$, and the fact that replacing
$n$ by $n+1$ does not change the exponential growth rate. Since
\[
\rho<h_{\top}(\sigma,\Sigma)-c_0,
\]
we conclude that
\[
\limsup_{n\to\infty}
\frac{1}{n}\log\sharp P_n^w
<
h_{\top}(\sigma,\Sigma),
\]
which completes the proof of the proposition.
\end{proof}

\section{Periodic orbits visiting the magic cylinder}\label{see}

{
We now count the periodic orbits that admit a symbolic lift on $\Sigma$ seeing the magic word $w$. Such orbits admit quotient lifts in the SPR component $Z$, and Theorem 5.3 of \cite{buzzi2020} guarantees that $\pi_N$ is injective on the set of quotient lifts related to $w$. Recall the fixed magic couple $(w,i)$ and set
\[
V_0:=A_w
=
\left\{
v_i:
v\in\mathcal{L}_{|w|}(\Sigma^\#),\ v\sim w
\right\}
\in\mathcal V_N.
\]
Since $(w,i)$ is a magic couple and $N=\delta_{\mathrm{rec}}(\sim)$, we have
\[
\sharp V_0
=
\delta_\sim(w,i)
=
\delta_{\mathrm{rec}}(\sim)
=
N.
\]
}

\subsection{Dynamics on magic cylinder}

\begin{proposition}\label{see magic}
{
For any periodic point $x\in\mathcal{H}$ admitting a lift on $\Sigma$ that encounters the magic word $w$, every quotient lift
\[
\hat y\in\pi_N^{-1}(x)\cap Z^\#
\]
that is related to $w$ visits the cylinder $[V_0]_Z$.
}
\end{proposition}

\begin{proof}
{
Fix such a quotient lift $\hat y\in\pi_N^{-1}(x)\cap Z^\#$. Since $\hat y$ is related to $w$, there exist infinitely many integers $p\gs0$ and infinitely many integers $p\ls0$ such that
\[
w_j\in\hat y_{p+j}
\quad\text{for every }0\ls j<|w|.
\]
From Lemma \ref{unique path}, for all $n\in\mathbb Z$,
\[
\hat y_n
=
\left\{
z_n^\xi:
\xi\in\hat y_0
\right\},
\]
where $z^\xi=Q(\hat y,\xi)$ for each $\xi\in\hat y_0$. Thus, for every $\xi\in\hat y_0$, the finite word
\[
z_p^\xi\cdots z_{p+|w|-1}^\xi
\sim w.
\]
At the fixed magic position $i$, this implies
\[
\hat y_{p+i}\subseteq V_0.
\]
Since
\[
\sharp\hat y_{p+i}
=
\sharp V_0
=
N,
\]
we obtain
\[
\hat y_{p+i}=V_0.
\]
Hence the orbit of $\hat y$ visits $[V_0]_Z$.
}
\end{proof}

Recall that $Z$ is the irreducible component of $\Sigma_N$ containing {$V_0$}. Proposition \ref{see magic} immediately yields:

\begin{corollary}\label{injective per}
For every $n$-periodic point $x\in\mathcal{H}$ with a symbolic lift on $\Sigma$ seeing the magic word $w$, there exists a unique $n$-periodic point $\hat y\in Z$ that necessarily visits the cylinder {$[V_0]_Z$}.
\end{corollary}

{ 
For a finite word}
\[
\underline a=a_0\cdots a_{r-1},
\qquad r\gs1,
\]
write $|\underline a|=r$. Define the return-word alphabet
\[
\bar{\mathcal V}
=
\left\{
[\underline a]:
\begin{array}{l}
|\underline a|\gs1,\quad
a_j=V_0\Longleftrightarrow j=0,\\[1mm]
[a_0\cdots a_{r-1}V_0]_Z\neq\emptyset
\end{array}
\right\},
\]
and the induced one-sided full shift
\[
\bar\Sigma
=
\bar{\mathcal V}^{\mathbb N\cup\{0\}}.
\]
Denote the left shift on $\bar\Sigma$ by $\bar\sigma$. Concatenating the
return words determines the nonnegative coordinates of a point in
$[V_0]_Z$. Choose an arbitrary compatible negative itinerary and denote the
resulting Borel choice by
\[
\bar p:\bar\Sigma\longrightarrow[V_0]_Z.
\]
Thus, for
\[
\bar x=
\bigl(
[\underline a^{\,0}],
[\underline a^{\,1}],
\dots
\bigr),
\]
the nonnegative coordinates of $\bar p(\bar x)$ are
\[
\underline a^{\,0}\underline a^{\,1}\cdots.
\]
Let
\[
H_{V_0}(\hat x)
=
\mathbf 1_{[V_0]_Z}(\hat x)
\min
\left\{
n\gs1:
\sigma_Z^n(\hat x)\in[V_0]_Z
\right\},
\qquad
\hat x\in Z,
\]
and set
\[
H_0(\bar x):=|\underline a^{\,0}|
=H_{V_0}\bigl(\bar p(\bar x)\bigr).
\]
For every function $u:Z\to\mathbb R$, define the induced function
\[
\bar u
=
\left(
\sum_{k=0}^{H_0-1}
u\circ\sigma_Z^k
\right)
\circ\bar p.
\]
Note that $H_0$ is locally H\"older continuous since it depends only on the zero-th coordinate.

{
Let
\[
h
:=
h_\mu(f)
=
h_{\mathrm{top}}(f,\H)
=
h_{\mathrm{top}}(\sigma_Z,Z)
=
P_G(\Sigma,0).
\]
Since
\[
h
=
\sup
\left\{
h_{\widetilde\nu}(\sigma_Z):
\widetilde\nu\in\mathcal M(\sigma_Z,Z)
\right\},
\]
we may restrict to measures $\widetilde\nu$ obtained from some
\[
\bar\nu\in\mathcal M(\bar\sigma,\bar\Sigma).
\]
To interpret the two-sided lift, let $\bar\nu^\natural$ be the natural
extension of $\bar\nu$ to
\[
\bar\Sigma^\natural:=\bar{\mathcal V}^{\mathbb Z},
\]
and let
\[
\bar p^\natural:\bar\Sigma^\natural\longrightarrow[V_0]_Z
\]
be the bi-infinite concatenation of return words. Then
\[
\widetilde\nu
=
\frac{1}{\int H_0\,\td\bar\nu}
\sum_{k=0}^{\infty}
(\sigma_Z^k\circ\bar p^\natural)_*
\left(
\mathbf 1_{\{H_0>k\}}\bar\nu^\natural
\right).
\]
Abramov's theorem \cite{LM} gives
\[
h_{\widetilde\nu}(\sigma_Z)
=
\frac{h_{\bar\nu}(\bar\sigma)}
{\int H_0\,\td\bar\nu}\ls h.
\]
Recall that  $\mu_Z$ is the MME on $Z$, which induces $\bar{\sigma}$-invariant measure $\bar{\mu}$ on $\bar{\Sigma}$.  Therefore,
\[
h_{\bar\mu}(\bar\sigma)
-
h\int H_0\,\td\bar\mu=
0,
\]
which meanwhile yields \[
P_G(\bar\Sigma,-hH_0)
:=
\sup
\left\{
h_{\bar\nu}(\bar\sigma)
-
h\int H_0\,\td\bar\nu:
\bar\nu\in\mathcal M(\bar\sigma,\bar\Sigma)
\right\}
=
0.
\]
That is, the MME $\mu_Z$ on $Z$ projecting to $\mu$ corresponds to the equilibrium state $\bar\mu$ for $-hH_0$ on $\bar\Sigma$. Since $\bar\Sigma$ is a full shift and $H_0$ depends only on the zero-th coordinate,
\[
\sum_{x\in\operatorname{Fix}_1(\bar\sigma)}
\e^{-hH_0(x)}
=
\e^{P_G(\bar\Sigma,-hH_0)}
=
1.
\]
}

{
For $s\in\mathbb C$ near $1$, let
\[
\mathcal L_s
:=
\overline{\mathcal L}_{-shH_0}
\]
be the Ruelle transfer operator on functions $g:\bar\Sigma\to\mathbb C$:
\[
(\mathcal L_sg)(x)
=
\sum_{\bar\sigma(y)=x}
\e^{-shH_0(y)}g(y).
\]
Define a metric on $\bar\Sigma$ by
\[
d_1(x,y)
=
\e^{-\min\{n:x_n\neq y_n\}},
\]
and, for $g\in C(\bar\Sigma)$, let
\[
\|g\|_\infty
=
\sup_{x\in\bar\Sigma}|g(x)|,
\qquad
Dg
=
\sup
\left\{
\frac{|g(x)-g(y)|}{d_1(x,y)}:
x\neq y
\right\}.
\]
Let $\operatorname{Lip}_{\mathbb C}(\bar\Sigma)$ be the Banach space of functions $g\in C(\bar\Sigma)$ with norm
\[
\|g\|_{\operatorname{Lip}}
=
\|g\|_\infty+Dg
<
\infty.
\]
}


{
Since $\bar\mu$ is the equilibrium state of $-hH_0$ on $\bar\Sigma$, it follows from \cite[Theorem 1.2]{Buzzi-Sarig} and \cite[Theorem 4]{OS} that $-hH_0$ is positive recurrent. Then, by \cite[Lemma 4]{OS.1}, the spectrum of
\[
\mathcal L_1:
\operatorname{Lip}_{\mathbb C}(\bar\Sigma)
\longrightarrow
\operatorname{Lip}_{\mathbb C}(\bar\Sigma)
\]
consists of a simple eigenvalue $1$ and a subset of
\[
\{z:|z|<\tau\}
\]
for some constant $\tau<1$. Furthermore, by virtue of the SPR property, $\mathcal L_s$ depends holomorphically on $s$; see \cite[Lemma 6]{OS.1}. Consequently, for every $s\in\mathbb C$ sufficiently close to $1$, the operator $\mathcal L_s$ admits a simple eigenvalue $\lambda(s)$ and a corresponding eigenfunction $g_s$, both depending holomorphically on $s$. For real $s$ near $1$,
\[
\lambda(s)
=
\e^{P_G(\bar\Sigma,-shH_0)}.
\]
After shrinking the neighborhood if necessary, choose the holomorphic logarithm $p(s)$ with $p(1)=0$ so that
\[
\lambda(s)=\e^{p(s)}.
\]
}

\subsection{Dynamical zeta function and meromorphic extension}

{
For a function $\Phi:\bar\Sigma\to\mathbb C$, write
\[
S_n\Phi(x)
:=
\sum_{j=0}^{n-1}
\Phi(\bar\sigma^jx).
\]
}

\begin{definition}[Dynamical Zeta Function]
{
The dynamical zeta function of the shift $\bar\sigma$ on $\bar\Sigma$ with respect to the function $-hH_0$ is defined as
\[
\zeta(s)
=
\prod_\gamma
\frac{1}
{1-\e^{-sh\ell(\gamma)}},
\qquad
\Re(s)>1,
\]
where the product is taken over all primitive periodic orbits $\gamma$ and
\[
\ell(\gamma)
:=
S_{|\gamma|}H_0(x)
\]
for some, and hence any, $x\in\gamma$. Here $|\gamma|$ is the period of $\gamma$ under $\bar\sigma$, while $\ell(\gamma)$ is its corresponding return-time length in $Z$.
}
\end{definition}

{
Equivalently,
\[
\zeta(s)
=
\exp
\left[
\sum_{n=1}^\infty
\frac1n
\sum_{x\in\operatorname{Fix}_n(\bar\sigma)}
\exp\bigl(-shS_nH_0(x)\bigr)
\right],
\]
where
\[
\operatorname{Fix}_n(\bar\sigma)
:=
\left\{
x\in\bar\Sigma:
\bar\sigma^n(x)=x
\right\}.
\]
Note that
\[
\zeta\left(s+\i\frac{2\pi}{h}\right)
=
\zeta(s).
\]

For a complex-valued function $\Phi:\bar\Sigma\to\mathbb C$, let
\[
Q_1(z,\Phi)
=
\exp
\left[
-\sum_{n=1}^\infty
\frac{z^n}{n}
\sum_{x\in\operatorname{Fix}_n(\bar\sigma)}
\exp\bigl(S_n\Phi(x)\bigr)
\right].
\]
If $\Phi$ depends only on the zero-th coordinate, then
\[
Q_1(z,\Phi)
=
1
-
z\sum_{x\in\operatorname{Fix}_1(\bar\sigma)}
\exp\bigl(\Phi(x)\bigr).
\]
Thus,
\[
\zeta(s)
=
\frac{1}{Q_1(1,-shH_0)}
=
\frac{1}
{1-\displaystyle\sum_{x\in\operatorname{Fix}_1(\bar\sigma)}
\exp\bigl(-shH_0(x)\bigr)}.
\]
For $\Re(s)>1$, the series in the denominator converges absolutely, since
\[
\sum_{x\in\operatorname{Fix}_1(\bar\sigma)}
\e^{-\Re(s)hH_0(x)}
<
\sum_{x\in\operatorname{Fix}_1(\bar\sigma)}
\e^{-hH_0(x)}
=
1,
\]
and hence $\zeta(s)$ is holomorphic on $\Re(s)>1$.
}

{
For each periodic orbit $\gamma$ of $(\bar\Sigma,\bar\sigma)$, define
\[
\mathcal N(\gamma)
:=
\e^{h\ell(\gamma)}.
\]
A fictitious orbit is a formal product
\[
\gamma'
=
\gamma_1^{n_1}\cdots\gamma_m^{n_m},
\]
where each $\gamma_j$ is a primitive periodic orbit and $n_j\in\mathbb N$. For such composite fictitious orbits, extend the weight by
\[
\mathcal N(\gamma')
=
\mathcal N(\gamma_1)^{n_1}
\cdots
\mathcal N(\gamma_m)^{n_m}.
\]
Define the \textit{dynamical von Mangoldt function} by
\[
\Lambda_{\mathrm{dyn}}(\gamma')
=
\begin{cases}
\log\mathcal N(\gamma),
&
\text{if }\gamma'=\gamma^l
\text{ for a primitive }\gamma,\ l\gs1,\\
0,
&
\text{otherwise}.
\end{cases}
\]
The dynamical zeta function is given by the Euler product over primitive orbits:
\[
\zeta(s)
=
\prod_\gamma
\left(
1-\mathcal N(\gamma)^{-s}
\right)^{-1}
=
\exp
\left(
\sum_\gamma
\sum_{k=1}^\infty
\frac1k
\mathcal N(\gamma)^{-ks}
\right),
\]
convergent for $\Re(s)>1$.
}

\begin{lemma}\label{ext}
$\zeta(s)$ admits a meromorphic extension to a neighborhood of $\{\Re(s)\gs1\}$ with simple poles at
\[
s
=
1+\i\frac{2\pi}{h}k,
\qquad
k\in\mathbb Z.
\]
\end{lemma}

\begin{proof}
{
Set
\[
\phi_s:=-shH_0.
\]
For $s\in\mathbb C$ near $1$, the Ruelle transfer operator $\mathcal L_s$
admits the simple eigenvalue
\[
\lambda(s)=\e^{p(s)}.
\]
The corresponding eigenfunction $g_s$ satisfies
\begin{equation}\label{eigen}
\mathcal L_sg_s
=
\e^{p(s)}g_s.
\end{equation}
Since $\bar\Sigma$ is a full shift and $H_0$ depends only on the zero-th
coordinate, the constant function is an eigenfunction and hence
\[
\e^{p(s)}
=
\sum_{x\in\operatorname{Fix}_1(\bar\sigma)}
\e^{\phi_s(x)}.
\]
Consequently,
\[
\zeta(s)
=
\frac{1}
{1-\displaystyle\sum_{x\in\operatorname{Fix}_1(\bar\sigma)}
\e^{\phi_s(x)}}
=
\frac{1}{1-\e^{p(s)}}.
\]

For $n\gs1$, set
\[
c_n
:=
\sharp
\left\{
x\in\operatorname{Fix}_1(\bar\sigma):
H_0(x)=n
\right\}.
\]
By the first-return construction,
\[
c_n
=
\mathrm{F}^*_n([V_0]_Z,0,\sigma_Z).
\]
Since $Z$ is SPR and $h_{\mathrm{top}}(\sigma_Z,Z)=h$,
\[
\limsup_{n\to\infty}
\frac1n\log c_n
<
h.
\]
Thus, for some $\delta>0$,
\[
F(s)
:=
\sum_{x\in\operatorname{Fix}_1(\bar\sigma)}
\e^{-shH_0(x)}
=
\sum_{n\gs1}c_n\e^{-shn}
\]
converges locally uniformly and is holomorphic on
$\Re(s)>1-\delta$; in particular, convergence and holomorphy on
$\Re(s)>1$ are automatic. Hence
\[
\zeta(s)=\frac{1}{1-F(s)}
\]
is meromorphic on $\Re(s)>1-\delta$.
}

\smallskip
\smallskip

We next split the proof into two steps:

\bigskip

\noindent{\bf Step 1:} $s=1$ is a simple pole of $\zeta(s)$.

\smallskip
\smallskip

{
We differentiate both sides of \eqref{eigen} with respect to $s$. For the left-hand side, using
\[
\frac{\td}{\td s}\mathcal L_s
=
\mathcal L_s\circ\frac{\td\phi_s}{\td s},
\]
we obtain
\[
\frac{\td}{\td s}
\bigl(
\mathcal L_sg_s
\bigr)
=
\mathcal L_s
\left(
\frac{\td\phi_s}{\td s}g_s
+
\frac{\td g_s}{\td s}
\right).
\]
Since
\[
\frac{\td\phi_s}{\td s}
=
-hH_0,
\]
this becomes
\[
\frac{\td}{\td s}
\bigl(
\mathcal L_sg_s
\bigr)
=
\mathcal L_s
\left(
-hH_0g_s
+
\frac{\td g_s}{\td s}
\right).
\]
For the right-hand side of \eqref{eigen}, the product rule gives
\[
\frac{\td}{\td s}
\left(
\e^{p(s)}g_s
\right)
=
\e^{p(s)}
\left(
p'(s)g_s
+
\frac{\td g_s}{\td s}
\right).
\]
Equating the derivatives yields
\begin{equation}\label{A}
\mathcal L_s
\left(
-hH_0g_s
+
\frac{\td g_s}{\td s}
\right)
=
\e^{p(s)}
\left(
p'(s)g_s
+
\frac{\td g_s}{\td s}
\right).
\end{equation}

We now evaluate \eqref{A} at $s=1$. Recall that $\e^{p(1)}=1$. Let $g_1$ be the eigenfunction of $\mathcal L_1$, satisfying
\[
\mathcal L_1g_1=g_1.
\]
The equilibrium state $\bar\mu$ is related to the eigenmeasure $m_1$ of $\mathcal L_1^*$ by
\[
\td\bar\mu
=
g_1\,\td m_1,
\]
and
\[
\mathcal L_1^*m_1=m_1.
\]
Substituting $s=1$ into \eqref{A}, we obtain
\begin{equation}\label{B}
\mathcal L_1
\left(
-hH_0g_1+g_1'
\right)
=
p'(1)g_1+g_1',
\end{equation}
where
\[
g_1'
=
\left.
\frac{\td g_s}{\td s}
\right|_{s=1}.
\]
Integrating both sides of \eqref{B} against $m_1$ and using
\[
\int_{\bar\Sigma}
\mathcal L_1\varphi\,\td m_1
=
\int_{\bar\Sigma}
\varphi\,\td m_1,
\]
we obtain
\[
\int_{\bar\Sigma}
\left(
-hH_0g_1+g_1'
\right)
\,\td m_1
=
p'(1)
\int_{\bar\Sigma}
g_1\,\td m_1
+
\int_{\bar\Sigma}
g_1'\,\td m_1.
\]
The terms involving $g_1'$ cancel. Since
\[
\int_{\bar\Sigma}
g_1\,\td m_1
=
1,
\]
we obtain
\[
p'(1)
=
-h
\int_{\bar\Sigma}
H_0\,\td\bar\mu.
\]
Consequently,
\[
\left.
\frac{\td}{\td s}
\e^{p(s)}
\right|_{s=1}
=
-h
\int_{\bar\Sigma}
H_0\,\td\bar\mu
\neq0.
\]
Hence $s=1$ is a simple pole.
}

\bigskip

\noindent{\bf Step 2:} $s=1+\i\frac{2\pi}{h}k$ for $k\in\mathbb Z$ are the only poles of $\zeta(s)$.

\smallskip
\smallskip

{
Since $(Z,\sigma_Z)$ is topologically mixing, the greatest common divisor of the first-return times to $[V_0]_Z$ is equal to $1$. Under the induced coding, these first-return times are precisely
\[
\left\{
H_0(x):
x\in\operatorname{Fix}_1(\bar\sigma)
\right\}.
\]
Hence
\[
\gcd
\left\{
H_0(x):
x\in\operatorname{Fix}_1(\bar\sigma)
\right\}
=
1.
\]
Furthermore,
\[
\zeta\left(
s+\i\frac{2\pi}{h}
\right)
=
\zeta(s),
\]
which implies that
\[
s
=
1+\i k\frac{2\pi}{h},
\qquad
k\in\mathbb Z,
\]
are also simple poles of $\zeta(s)$.

Now suppose
\[
s
=
1+\i c/h
\]
is a singularity of $\zeta(s)$. We have
\[
\left|
\sum_{x\in\operatorname{Fix}_1(\bar\sigma)}
\e^{-\i cH_0(x)}
\e^{-hH_0(x)}
\right|
\ls
\sum_{x\in\operatorname{Fix}_1(\bar\sigma)}
\e^{-hH_0(x)}
=
1.
\]
If the inequality is strict, then $s=1+\i c/h$ cannot be a singularity of $\zeta(s)$. Consequently, there exists $\theta\in[0,2\pi)$ such that
\[
\sum_{x\in\operatorname{Fix}_1(\bar\sigma)}
\e^{-\i cH_0(x)}
\e^{-hH_0(x)}
=
\e^{\i\theta}.
\]
In fact, $\e^{\i\theta}=1$; otherwise, $s=1+\i c/h$ is not a singularity of $\zeta(s)$. Therefore,
\[
\sum_{x\in\operatorname{Fix}_1(\bar\sigma)}
\e^{-\i cH_0(x)}
\e^{-hH_0(x)}
=
1.
\]
Combined with
\[
\sum_{x\in\operatorname{Fix}_1(\bar\sigma)}
\e^{-hH_0(x)}
=
1,
\]
convexity yields
\[
\e^{-\i cH_0(x)}
=
1
\quad
\text{for every }
x\in\operatorname{Fix}_1(\bar\sigma).
\]
Since the greatest common divisor of the values $H_0(x)$ is $1$, $c$ is an integer multiple of $2\pi$. Hence the only poles of $\zeta(s)$ on the line $\Re(s)=1$ are
\[
s
=
1+\i k\frac{2\pi}{h},
\qquad
k\in\mathbb Z.
\]
Since $F(s)$ is holomorphic and $2\pi\i/h$-periodic on
$\Re(s)>1-\delta$, these boundary zeros of $1-F(s)$ are simple and
isolated. By compactness of one fundamental strip, after decreasing
$\delta>0$ if necessary, $1-F(s)$ has no other zeros in
$\Re(s)\gs1-\delta$. Thus these are the only poles in a neighborhood of
$\{\Re(s)\gs1\}$.
}
\end{proof}

\subsection{Counting on periodic orbits}

{
For $R\geqslant1$, define
\[
\Psi(R)
:=
\sum_{\mathcal N(\gamma')\ls R}
\Lambda_{\mathrm{dyn}}(\gamma').
\]
Then
\[
\begin{aligned}
\frac{\zeta'(s)}{\zeta(s)}
&=
(\log\zeta(s))'\\[2mm]
&=
-\sum_\gamma
\log\mathcal N(\gamma)
\sum_{k=1}^\infty
\mathcal N(\gamma)^{-ks}\\[2mm]
&=
-\sum_{\gamma'}
\frac{\Lambda_{\mathrm{dyn}}(\gamma')}
{\mathcal N(\gamma')^s}\\[2mm]
&=
-\int_1^\infty
R^{-s}\,\td\Psi(R).
\end{aligned}
\]
Let
\[
b:=\e^h.
\]
Then $\mathcal N(\gamma')$ is a positive integer power of $b$. Thus,
\begin{equation}\label{z1}
\frac{\zeta'(s)}{\zeta(s)}
=
-\sum_{n=1}^\infty
b^{-ns}
\sum_{\mathcal N(\gamma')=b^n}
\Lambda_{\mathrm{dyn}}(\gamma').
\end{equation}
}

By Lemma \ref{ext}, the function $\zeta'(s)/\zeta(s)$ has simple poles at
\[
s
=
1+\i k(2\pi/h),
\qquad
k\in\mathbb Z,
\]
with residue $-1$. Therefore,
\begin{equation}\label{z2}
\frac{\zeta'(s)}{\zeta(s)}
=
\frac{-h}{1-\e^{-h(s-1)}}
+
v(s)
=
-h
\sum_{n=0}^\infty
\e^{nh}\e^{-nhs}
+
v(s),
\end{equation}
where $v(s)$ is analytic on $\Re(s)>1-\delta$ for some $\delta>0$.
{
Moreover, by \eqref{z2} and the $2\pi\i/h$-periodicity of both
$\zeta'(s)/\zeta(s)$ and the principal term, $v(s)$ is
$2\pi\i/h$-periodic.
}

\begin{lemma}\label{Zest}
{
\[
\Psi(R)
\sim
h\sum_{\substack{n\gs1\\ \e^{nh}\ls R}}\e^{nh},
\]
where $u_R\sim w_R$ means
\[
\lim_{R\to\infty}
\frac{u_R}{w_R}
=
1.
\]
}
\end{lemma}

\begin{proof}
{
By the $2\pi\i/h$-periodicity of $v(s)$, the substitution
\[
z=b^{-s}
\]
defines a single-valued holomorphic function $V(z)$ on the punctured disk
\[
0<|z|<b^{-1+\delta}.
\]
Moreover, as $\Re(s)\to+\infty$, one has $\zeta'(s)/\zeta(s)\to0$ and the
principal term in \eqref{z2} tends to $-h$, so $v(s)\to h$. Thus the
singularity of $V$ at $z=0$ is removable, and
\[
v(s)
=
\sum_{n=0}^\infty
a_nb^{-ns}.
\]
Equating \eqref{z1} and \eqref{z2} and matching the coefficients of
$b^{-ns}$ for all $n\gs1$, we obtain
\[
-
\sum_{\mathcal N(\gamma')=b^n}
\Lambda_{\mathrm{dyn}}(\gamma')
=
-hb^n+a_n,
\]
and hence
\[
\sum_{\mathcal N(\gamma')=b^n}
\Lambda_{\mathrm{dyn}}(\gamma')
-
h\e^{nh}
=
-a_n.
\]
The radius of convergence $R_0$ of
\[
\sum_{n=0}^\infty a_nz^n
\]
satisfies
\[
R_0\geqslant b^{-1+\delta}
>
b^{-1+\delta/2}.
\]
By the Cauchy--Hadamard formula,
\[
\limsup_{n\to\infty}
\sqrt[n]{|a_n|}
<
b^{1-\delta/2}.
\]
Therefore, there exists $C>0$ such that
\[
|a_n|
\ls
Cb^{n(1-\delta/4)}
=
C\e^{nh(1-\delta/4)}
\]
for all $n\gs1$.

Finally, applying the Stolz--Ces\`aro theorem to the sequence of cutoffs $R=b^m$, we obtain
\[
\begin{aligned}
\lim_{m\to\infty}
\frac{\Psi(b^m)}
{h\displaystyle\sum_{1\ls n\ls m}\e^{nh}}
&=
\lim_{m\to\infty}
\frac{
\displaystyle\sum_{\mathcal N(\gamma')=b^m}
\Lambda_{\mathrm{dyn}}(\gamma')
}
{h\e^{mh}}\\[2mm]
&=
\lim_{m\to\infty}
\frac{h\e^{mh}-a_m}
{h\e^{mh}}
=
1.
\end{aligned}
\]
Since $\mathcal N(\gamma')$ takes only the values $b^n$, both
$\Psi(R)$ and $\sum_{\substack{n\gs1\\ \e^{nh}\ls R}}\e^{nh}$ are constant between
successive powers of $b$. Hence the same limit holds for arbitrary
$R\to\infty$. This proves the lemma.
}
\end{proof}

{
Define the periodic-orbit counting function
\[
\Pi(R)
:=
\sharp
\left\{
\gamma\text{ primitive}:
\mathcal N(\gamma)\ls R
\right\}.
\]
}

\begin{lemma}\label{asy eq}
{
\[
\Pi(R)
\sim
\frac{\Psi(R)}{\log R}.
\]
}
\end{lemma}

\begin{proof}
{
By the definition of the dynamical von Mangoldt function,
\[
\Psi(R)
=
\sum_{\substack{\gamma\ {\rm primitive}\\
\mathcal N(\gamma)\ls R}}
k_R(\gamma)\log\mathcal N(\gamma),
\]
where
\[
k_R(\gamma)
:=
\left\lfloor
\frac{\log R}{\log\mathcal N(\gamma)}
\right\rfloor
\]
is the maximal positive integer such that
\[
\mathcal N(\gamma)^{k_R(\gamma)}\ls R.
\]
Thus
\[
\log\mathcal N(\gamma)
\ls
k_R(\gamma)\log\mathcal N(\gamma)
\ls
\log R.
\]
Summing over all primitive orbits with $\mathcal N(\gamma)\ls R$ gives
\[
\Psi(R)
\ls
\sum_{\substack{\gamma\ {\rm primitive}\\
\mathcal N(\gamma)\ls R}}
\log R
=
\Pi(R)\log R.
\]

We next use the Euler product to prove
\[
\frac{\Pi(R)}{R^q}
\longrightarrow
0
\]
for every real $q>1$. Choose $1<q'<q$. Then
\[
\begin{aligned}
\zeta(q')
&=
\prod_\gamma
\left(
1-\mathcal N(\gamma)^{-q'}
\right)^{-1}\\[2mm]
&\gs
\prod_{\mathcal N(\gamma)\ls R}
\left(
1-R^{-q'}
\right)^{-1}\\[2mm]
&=
\left(
1-R^{-q'}
\right)^{-\Pi(R)}.
\end{aligned}
\]
Taking logarithms yields
\[
\frac{\log\zeta(q')}
{R^{q-q'}}
\gs
\frac{\Pi(R)}{R^q},
\]
and hence
\[
\frac{\Pi(R)}{R^q}
\longrightarrow0.
\]

Next fix $0<\theta<1$ and put
\[
y
=
\left(
\frac{R}{\log R}
\right)^\theta.
\]
Then
\[
\begin{aligned}
\Pi(R)
&=
\Pi(y)
+
\sum_{\substack{\gamma\\
y<\mathcal N(\gamma)\ls R}}
1\\[2mm]
&\ls
\Pi(y)
+
\sum_{\substack{\gamma\\
\mathcal N(\gamma)\ls R}}
\frac{\log\mathcal N(\gamma)}
{\log y}\\[2mm]
&\ls
\Pi(y)
+
\frac{\Psi(R)}{\log y}\\[2mm]
&=
\Pi(y)
+
\frac{\Psi(R)}
{\theta(\log R-\log\log R)}.
\end{aligned}
\]
Therefore,
\[
\varlimsup_{R\to\infty}
\frac{\Pi(R)\log R}
{\Psi(R)}
\ls
\varlimsup_{R\to\infty}
\frac{\Pi(y)\log R}
{\Psi(R)}
+
\frac1\theta.
\]
Since
\[
\frac{\Pi(y)}{y^{1/\theta}}
\longrightarrow0
\]
and Lemma \ref{Zest} implies that
\[
\frac{R}{\Psi(R)}
\]
is bounded from above, we have
\[
\begin{aligned}
\varlimsup_{R\to\infty}
\frac{\Pi(y)\log R}
{\Psi(R)}
&=
\varlimsup_{R\to\infty}
\left(
\frac{\Pi(y)}{y^{1/\theta}}
\cdot
\frac{R}{\Psi(R)}
\cdot
\frac{y^{1/\theta}\log R}{R}
\right)\\[2mm]
&=
0,
\end{aligned}
\]
where
\[
\frac{y^{1/\theta}\log R}{R}
=
1.
\]
Thus,
\[
\varlimsup_{R\to\infty}
\frac{\Pi(R)\log R}
{\Psi(R)}
\ls
\frac1\theta.
\]
Letting $\theta\to1^-$, we obtain
\[
\varlimsup_{R\to\infty}
\frac{\Pi(R)\log R}
{\Psi(R)}
\ls
1.
\]
Together with
\[
\Psi(R)\ls\Pi(R)\log R,
\]
this gives
\[
\Psi(R)
\sim
\Pi(R)\log R,
\]
or equivalently
\[
\Pi(R)
\sim
\frac{\Psi(R)}{\log R}.
\]
}
\end{proof}

{
For each integer $n\gs1$, define
\[
\mathfrak P_{\leq n}
:=
\left\{
\text{primitive periodic }\gamma:
\ell(\gamma)\ls n
\right\},
\qquad
\Pi_{\leq n}
:=
\sharp\mathfrak P_{\leq n},
\]
and
\[
\Pi_n
:=
\Pi_{\leq n}
-
\Pi_{\leq n-1},
\qquad
\Pi_{\leq0}:=0.
\]
Combining Lemmas \ref{Zest} and \ref{asy eq}, we obtain
\[
\Pi_{\leq n}
=
\Pi(\e^{hn})
\sim
\frac{\Psi(\e^{hn})}{hn}
\sim
\frac{\displaystyle\sum_{1\ls j\ls n}\e^{jh}}{n}.
\]
Thus, for every $\epsilon>0$, there exists $n_0\in\mathbb N$ such that, for every integer $n\gs n_0$,
\[
(1-\epsilon)
\sum_{1\ls j\ls n}
\e^{jh}
\ls
n\Pi_{\leq n}
\ls
(1+\epsilon)
\sum_{1\ls j\ls n}
\e^{jh}.
\]
Each corresponding periodic orbit of return-time length $n$ contains exactly $n$ periodic points in $Z$. Hence
\[
n\Pi_n
=
n
\left(
\Pi_{\leq n}
-
\Pi_{\leq n-1}
\right)
\]
counts the periodic points $\hat x\in Z$ of period $n$ whose orbits visit $[V_0]_Z$.
For every $n\gs n_0+1$,
\[
\begin{aligned}
n\Pi_n
&=
n\Pi_{\leq n}
-
(n-1)\Pi_{\leq n-1}
-
\Pi_{\leq n-1}\\[2mm]
&\ls
(1+\epsilon)
\sum_{1\ls j\ls n}
\e^{jh}
-
(1-\epsilon)
\sum_{1\ls j\ls n-1}
\e^{jh}
-
\Pi_{\leq n-1}\\[2mm]
&\ls
\e^{nh}
+
2\epsilon
\frac{\e^h(\e^{nh}-1)}
{\e^h-1}
-
\frac{1-\epsilon}{n-1}
\frac{\e^h(\e^{(n-1)h}-1)}
{\e^h-1}.
\end{aligned}
\]
Dividing by $\e^{nh}$ and letting $n\to\infty$, we obtain
\[
\limsup_{n\to\infty}
\frac{n\Pi_n}{\e^{nh}}
\ls
1
+
2\epsilon
\frac{\e^h}{\e^h-1}.
\]
By the arbitrariness of $\epsilon$,
\[
\limsup_{n\to\infty}
\frac{n\Pi_n}{\e^{nh}}
\ls
1.
\]
}

{
The periodic points in
\[
P^c_{\chi_1,\chi_2}(f,\H,\eta,n)\setminus P_n^w
\]
admit a symbolic lift seeing $w$ and hence, by Corollary
\ref{injective per}, are injectively represented by $n$-periodic points
of $Z$ whose orbits visit $[V_0]_Z$. Therefore
\[
\sharp P^c_{\chi_1,\chi_2}(f,\H,\eta,n)
\ls
n\Pi_n+\sharp P_n^w.
\]
By Proposition \ref{not see magic},
\[
\frac{\sharp P_n^w}{\e^{nh}}
\longrightarrow0.
\]
Consequently,
\[
\limsup_{n\to\infty}
\frac{
\sharp P^c_{\chi_1,\chi_2}(f,\H,\eta,n)
}
{\e^{nh}}
\ls
\limsup_{n\to\infty}
\frac{n\Pi_n}{\e^{nh}}
\ls
1.
\]
}
Combining with Proposition \ref{away}, this implies
\[
\limsup_{n\to\infty}
\frac{
\sharp P_{\chi_1,\chi_2}(f,\H,n)
}
{\e^{n h_{\top}(f,\H)}}
\ls
1.
\]
Thus, the proof of Theorem \ref{exp} is complete.

\begin{proof}[Proof of Corollary \ref{transitive case}]
Since $f$ is transitive, there exists a unique homoclinic class $\H$ such that $h_{\top}(f,\mathcal{H}) = h_{\top}(f) > 0$. It follows that there exists a unique MME $\mu$ (supported on $\mathcal{H}$) with $h_{\mu}(f) = h_{\top}(f)$, see \cite[Theorem 1]{BCS}. From Proposition \ref{away}, we know that the periodic orbits away from $\mu$ in measure have a growth rate gap from $h_{\text{top}}(f)$. Thus, we only need to consider periodic orbits close to $\mu$ that are $(\chi_1,\chi_2)$-hyperbolic with $\chi_2$ chosen as in (\ref{chi1}). This implies that these periodic orbits are homoclinically related to $\mu$, and hence are contained in $P_{\chi_1,\chi_2}(f, \mathcal{H}, n)$. Therefore, by Theorem \ref{exp}, we have

$$\lim_{\substack{l \mid n \\ n \to \infty}} \frac{\sharp P_{\chi_1,\chi_2}(f, n)}{\e^{n h_{\text{top}}(f)}} = \lim_{\substack{l \mid n \\ n \to \infty}} \frac{\sharp P_{\chi_1,\chi_2}(f, \mathcal{H}, n)}{\e^{n h_{\text{top}}(f)}} = l,
$$
where $l$ is the period of $\mathcal{H}$. In particular, if $f$ is mixing, then $l = 1$.  
\end{proof}
\begingroup
\renewcommand{\addcontentsline}[3]{}

\section*{Data availability}
No data were used for the research described in this article.


\endgroup 


\bigskip



\bibliographystyle{abbrv}
{\footnotesize\bibliography{library}}

\end{document}